\documentclass[12pt,reqno]{amsart}
\usepackage{amsthm,mathrsfs,amsmath,amstext, amsxtra, amsfonts, amssymb, latexsym,array}
\usepackage{graphicx,xcolor}
\usepackage{lmodern}
\usepackage[colorlinks, linkcolor=red, citecolor=blue, urlcolor=blue, hypertexnames=false]{hyperref}

\newcommand{\norm}[1]{ \left\|  #1 \right\| }

\newcommand{\be}{\begin{equation}}
\newcommand{\ee}{\end{equation}}

\def\al{\alpha}
\def\De{\Delta}
\def\ga{\gamma}
\def\lam{\lambda}

\def\veps{\varepsilon}

\def\inv{^{-1}}
\def\iy{\infty}

\newcommand{\R}{\mathbb R}

\renewcommand{\theequation}{\thesection.\arabic{equation}}
\newtheorem{theorem}{Theorem}[section]
\newtheorem{lemma}[theorem]{Lemma}
\newtheorem{proposition}[theorem]{Proposition}
\newtheorem{corollary}[theorem]{Corollary}
\theoremstyle{definition}
\newtheorem{definition}[theorem]{Definition}
\newtheorem{remark}[theorem]{Remark}

\title[Orbital Stability for a fourth-order NLS]{ 
Orbital Stability of Standing Waves for a fourth-order nonlinear Schr\"odinger equation with mixed  dispersions}
\author{Tingjian Luo}
\address[Tingjian Luo]{School of Mathematics and Information Sciences, Guangzhou University, Guangzhou 510006, China}
\email{luotj@gzhu.edu.cn}
\author{Shijun Zheng}
\address[Shijun Zheng]{Department of Mathematical Sciences\\
Georgia Southern University\\
Statesboro, Georgia 30460-8093, United States}
\email{szheng@GeorgiaSouthern.edu}
\author{Shihui Zhu}
\address[Shihui Zhu]{School of Mathematical Sciences, Sichuan Normal University, Chengdu 610066, China}
\email{shihuizhumath@163.com;\ shihuizhumath@sicnu.edu.cn}
\date{\today}


\begin{document}
	
\begin{abstract} In this paper, we study the ground state standing wave solutions for the focusing bi-harmonic nonlinear Schr\"{o}dinger equation with a $\mu$-Laplacian term (BNLS). 
Such  BNLS models the propagation of intense laser beams in a bulk medium with a  second-order dispersion term. 
Denote by $Q_p$  the ground state for the BNLS with $\mu=0$. 
 We prove that in the mass-subcritical regime $p\in (1,1+\frac{8}{d})$, 
there exist orbitally stable {ground state solutions} for the BNLS
when $\mu\in ( -\lambda_0, \iy)$   % 
  for some $\lambda_0=\lambda_0(p, d,\|Q_p\|_{L^2})>0$.  
Moreover, in the mass-critical case $p=1+\frac{8}{d}$\,, we prove the orbital stability  on certain mass level below $\|Q^*\|_{L^2}$, 
provided $\mu\in (-\lam_1,0)$, where $\lam_1=\dfrac{4\|\nabla Q^*\|^2_{L^2}}{\|Q^*\|^2_{L^2}}$ and $Q^*=Q_{1+8/d}$. 
The proofs are mainly based on the profile decomposition and a sharp Gagliardo-Nirenberg type inequality.  
Our treatment allows to fill  the gap concerning existence of the ground states for the BNLS 
when $\mu$ is negative and $p\in (1,1+\frac8d]$. 
 \end{abstract}	
	
\maketitle

\section{Introduction}
\setcounter{equation}{0}
  Consider the following  fourth-order  nonlinear
Schr\"{o}dinger equation, called  bi-harmonic NLS (BNLS), with a focusing nonlinearity 
\begin{equation}\label{BNLS}
i\psi_t-\Delta^2
\psi +\mu \Delta \psi +|\psi|^{p-1}\psi=0,\quad t\geq 0,\ x\in \mathbb{R}^d,
\end{equation}
where $i=\sqrt{-1}$ and the parameter $\mu \in \R$; $\psi=\psi(t,x)$:
$\mathbb{R}\times \mathbb{R}^d \to \mathbb{C}$ is the complex-valued
wave function and $d$ is the spacial dimension;
$\Delta=\sum\limits_{j=1}^{d}\frac{\partial^2}{\partial x_j^2}$
is the Laplace operator in $\mathbb{R}^d$ and
 $\Delta^2$ is the
biharmonic operator;
 $1<p<\frac{2d}{(d-4)^+}-1$
(here by convention $\frac{2d}{(d-4)^+}=+\infty$ if
$d=1,2,3,4$; $\frac{2d}{(d-4)^+}=\frac{2d}{d-4}$ if $d\geq 5$). 
Impose the  initial condition for Eq.  (\ref{BNLS})
\begin{equation}
 \label{BNLS1} \psi(0,x)=\psi_{0}\in H^2, \quad x\in \mathbb{R}^d,
\end{equation}
where $H^2=H^2(\mathbb{R}^d):=\{v\in L^2\ | \ \nabla v\in L^2, \Delta v\in L^2 \}$ is the standard Sobolev space which serves as the energy space.  The BNLS type equation %(\ref{BNLS}) 
was  introduced in \cite{Karpman1996,Karpman2000} %Karpman and Shagalov
   where it  took into account the role of small 
fourth-order dispersion term in the propagation of intense laser 
beams in a bulk medium with Kerr nonlinearity.  
 The  case $\mu=0$ was  considered earlier in \cite{IvanKo83, Tur85} 
in the context of stability of solitons in magnetic materials when the effective quasi-particle mass becomes infinite.  In the setting of beam modeling for optical fibre, a nice derivation from the Helmholtz equation to the BNLS can be found in \cite{Fibich2002} where the nonparaxial effect contributes to the perturbed NLS
with a fourth-order correction term. 

 Recently, the fourth-order nonlinear Schr\"{o}dinger equations have received increasing attentions. 
 The local well-posedness  for the  Cauchy problem 
(\ref{BNLS})-(\ref{BNLS1}) in $H^2$ was obtained in  %Kenig, Ponce and Vega  
\cite{BenKochSaut2000,KenigPonceVega1991,Pausader2007DPDE}.  %Ben-Artzi, Koch and Saut   
  Fibich, Ilan and Papanicolaou %[Theorem 4.2, Theorem 6.1]
\cite{Fibich2002}  
 studied the global well-posedness for  (\ref{BNLS})-(\ref{BNLS1})  in $H^2$  
 in the  case where $\mu\ge 0$ and $p\in (1,1+\frac{8}{d}]$.
%\begin{enumerate}  
%\item  %(focusing) 
 % $\mu\ge 0$ and $p\in (1,1+\frac{8}{d})$
%\item %(focusing) 
% $\mu\ge 0$ and $p=1+\frac{8}{d}$,  $\norm{\psi_0}_{2}<(\frac{p+1}{2 B_{p,d}} )^{1/(p-1)}=\norm{Q^*}_2$ 
%\item  (defocusing) if $\mu=0$, $N(u)=+|u|^{p-1} u$ and $p\in (1, 1+\frac{8}{(d-4)^+})$  
%\end{enumerate}   where $B_{p,d}$ is an optimal constant for the $J$-functional (\ref{JpdU}).
  %\edz{there was a typo for the sign $\eps$ in \cite{Fibich2002}} 
 %Then one  can verify that if $1<p<1+\frac8d$,  the solution of the  Cauchy problem (\ref{BNLS})-(\ref{BNLS1}) is unique and exists globally in time by the suitable conservation laws
%\footnote{The exact analogue of Th.4.2 in \cite{Fibich2002} holds on a bounded domain $\Om$ with smooth boundary $\pa\Om$ for (\ref{BNLS}) with dirichlet b.c.  cf. \cite[Th.4.4]{Fibich2002}
%Further, for higher order NLS the analogue for gwp holds with similar conditions \cite[TH.5.1]{Fibich2002}
% }
%For $\mu=+1$, the same theorem is proven \cite[Th.6.1]{Fibich2002}  (There was a sign typo in the definition for the energy). %$H(0)$)
If  $p\geq1+\frac8d$\ ,   %\mu\in \R if p\in (1+\frac8d,1+\frac{8}{d-4})$
%and \mu\ge 0 if p= 1+\frac8d$ with E(u_0)<0
Boulenger and  Lenzmann  \cite{BL2015}  proved the existence of blowup  solutions for (\ref{BNLS})-(\ref{BNLS1}), also see related numerical results in \cite{BaruchFibichMandelbaum2010}. 
 This suggests that $p=1+\frac8d$  is  the critical exponent for global existence and blowup for   (\ref{BNLS}). 
 %and  $p\in (1,1+\frac8d)$ is a subcritical exponent 
%\cite{BCGJ2} for recent refined  %for (\ref{BNLS})  {Fibich2002} 
{The articles \cite{Miao2009JDE,  PauShao2010, Segata2006} studied the scattering for the fourth-order NLS. 
%with $p\ge 1+\frac8d$\ . %(\ref{BNLS} }  %Pausader \cite{Pausader2007DPDE} 
   
   In this paper we are concerned with   the existence of standing wave solutions and their stability 
   for (\ref{BNLS}). % p\in (1,1+\frac8d]
    Let $\mu, \omega\in \R$ and $u=u(x)$ be a solution of the following elliptic equation  
 \be\label{1.3} \Delta^2 u-\mu \Delta u+\omega u-|u|^{p-1}
 u=0,\ \ u\in H^2.\ee
 Then $\psi(t,x)=e^{i\omega t}u(x)$ is a {\em standing wave  solution} of  (\ref{BNLS}), see
   \cite{Cazenave2003,Levandosky1998, NataliPastor2015}. 
   Therefore, it is equivalent to study  the problem concerning the existence and stability  for the stationary solutions of (\ref{1.3}). 
For such problem the cases $\mu\ge 0$, $p\in (1,1+\frac{8}{(d-4)^+})$
were considered in \cite{BaruchFibichMandelbaum2010,BCMN,BCGJ,Segata2010,zhuzhangyang2010}.   
 Numerical studies for (\ref{BNLS}) can be found  in 
  \cite{Fibich2002,Karpman1996,Karpman2000},  \cite{BaruchFibichMandelbaum2010} and the references therein.  %   for $1<p\leq 1+\frac{4}{d}$, there exist stable solitons 
   When $\mu\ge 0$, $p\geq 1+\frac{8}{d}$, the evolution system (\ref{BNLS})-(\ref{BNLS1}) may have  
   strong instability that shows blowup properties  \cite{BF2011,BCGJ2,zhuzhangyang2010}. 
     %  remains a gap $1+\frac{4}{d}<p<1+\frac{8}{d}  
However, the case $\mu<0$ does not seem to have been well understood in the preceding literature. 
In such case, it remains open  how to construct {\em orbitally stable} standing waves,  
where  the operators $\Delta^2$ and $-\mu \Delta$ have counter-competing effect, which may lead to technical difficulties in the analysis and construction of stable states of \eqref{1.3}. 

In order to treat the case $\mu<0$ for the existence and stability of ground states for (\ref{1.3})
%and $p\in (1,1+\frac8d)$ 
   we consider the following minimization problem 
\begin{equation}\label{VP}
m_{\mu}:=\inf\limits_{u\in B_1} E_{\mu}(u),  \tag{VP}
\end{equation}
where   $B_1:=\{u\in H^2\ |\ \int |u|^2 dx = 1 \}$ and 
\begin{align}\label{E:4NLS-mu}
 E_{\mu}(u):= \frac 12 \int|\Delta  u|^{2}dx+\frac{\mu}{2} \int|\nabla u|^{2}dx-\frac {1}{p+1} \int |u|^{p+1}dx
 \end{align} is the associated energy for  \eqref{1.3}.  
Denote the set of minimizers for (\ref{VP})
\begin{eqnarray}\label{3.18}
    \mathcal{M}_{\mu}:=\{u\in B_1 \;\vert \;  u\;\text{ is a  minimizer of (\ref{VP})}\}.
\end{eqnarray}
It is easy to see that for any $u\in \mathcal{M}_{\mu}$, there exists $\omega \in \mathbb{R}$, namely, a Lagrange multiplier, such that $(u, \omega)$ solves the equation (\ref{1.3}).
Since $u$ minimizes the energy $E_\mu$ on $B_1$, %level set %$\{ \norm{u}_2=1\}\cap H^2$, 
we call a minimizer of (\ref{VP}) a {\bf ground state solution} (g.s.s.) of (\ref{1.3}). The definition of orbital stability for $M_\mu$ is standard as given in Section \ref{s:prelimin}.

In Theorem \ref{th1} and Theorem \ref{th22}  we establish the existence and orbital stability for the set of ground states for (\ref{BNLS})-(\ref{BNLS1}) for suitable $\mu\in \R$ and $p\in (1,1+\frac8d]$. 
In particular, 
Theorem \ref{th1} shows that if $1<p<1+\frac8d$ and $\mu\in (-\lam_0,0)$ for some constant $\lam_0>0$,  
then there exist ground states for  (\ref{BNLS}), which are orbitally stable.   In the critical case $p=1+\frac8d$ and $\mu\in (-\lam_1,0)$ for some  constant $\lam_1>0$, 
we prove in Theorem \ref{th22} 
 that if  the initial data satisfies $\|\psi_0\|_{L^2}<\|Q^*\|_{L^2}$, then the  standing waves of  (\ref{BNLS}) are orbitally stable, where $Q^*=Q_{1+8/d}$ is a fixed ground state of %minimization solution of  the  $J$-
  \be\label{1.4}
 \Delta^2Q+ \frac4dQ-|Q|^{\frac8d}Q=0,\ \ Q\in H^2.
 \ee 
Since the uniqueness of solutions to (\ref{1.4}) is not determined yet in the literature, throughout this paper
we will choose $Q^*$ to be a fixed  minimizer for the $J$-functional (\ref{JpdU}). The same choice applies to $Q_p:=Q_{p,d}$ for other values of $p$, see (\ref{1.5}).
Note that, in view of (\ref{B-L2critic}),  the $L^2$-norm of $Q_p$ is independent of the choices of any particular $J$-minimizer. % actual ground state for (\ref{BNLS}) is the ground state of  (\ref{1.3}
The threshold $\|Q^*\|_{L^2}$ is sharp in virtue of the instability  result  in \cite{BaruchFibichMandelbaum2010,BL2015,Fibich2002} 
%implies that there exist finite time blow-up solutions 
for  (\ref{BNLS})-(\ref{BNLS1}) with $\mu\ge 0$ and $p=1+\frac8d$
if the initial data $\|\psi_0\|_{L^2}\geq \|Q^*\|_{L^2}$. 

To prove the existence of the set $\mathcal{M}_\mu$ 
for a general range $\mu\in \R$ and \mbox{$p\in (1,1+\frac8d]$,} we mainly apply the profile decomposition method 
(Lemma \ref{decomposition}), along with  the Gagliardo-Nirenberg type inequality (Lemma \ref{prop2.4}),
 to show the concentration compactness: Any minimizing sequence for (\ref{VP}) or (\ref{1.7}) has a subsequence converging to a g.s.s. of (\ref{1.3}) modular translations.  
Then the stability of standing waves for (\ref{BNLS}) will follow via the standard arguments as given in Section \ref{s:Q-mu_subcri}. 
The profile decomposition  was initially proposed in G\'{e}rard  \cite{Gerard1998}. Hmidi and Keraani \cite{HmidiKeraani2005} obtained the profile decomposition of bounded sequences in $H^1$ %and gave a new and simple proof of some dynamical properties of blow-up solutions 
to deal with the classical second-order NLS.  In  \cite{zhuzhangyang2010}, Zhu et al. obtained the profile decomposition of bounded sequences in $H^2$. % to study the limiting profile of blow-up solution 

Now let us precisely elaborate these two theorems. %$m_{\mu}\neq -\infty$  provided $1<p<1+\frac{8}{d}$ %  Euler-Lagrange equation   u\in \mathcal{M}_{\mu}  exist \om\in R s.t  u solves (\ref{1.3})
In Section \ref{s:Q-mu_subcri}, we consider the $L^2$-subcritical case: $1<p<1+\frac8d$. By scaling argument and (\ref{E:4NLS-mu}), we  observe  that $m_{\mu}\in(-\infty,  0]$ is continuous and non-decreasing with respect to $\mu\in \mathbb{R}$  (Lemma \ref{lm3.1} and  Lemma \ref{lm3.0}). Denote
\begin{eqnarray}\label{1.7.1}
\mu_0:=\mu_{0,p}=\sup \{\mu>0\ | \ m_{\mu}<0 \}.
\end{eqnarray}
Then, according to Lemma \ref{lm3.2},
\begin{align*}
\begin{cases} 
\mu_0=\iy & \text{if}\; p\in (1,1+\frac4d)\\ 
0<\mu_0<\iy & \text{if}\; p\in [1+\frac4d,1+\frac8d).
\end{cases}\end{align*}
 Let $Q_p=Q_{p,d}$ be a fixed ground state of the following bi-harmonic equation
\begin{eqnarray}\label{1.5}
\frac{(p-1)d}{8}\Delta^2
Q+\Big[1+\frac{(p-1)(4-d)}{8}\Big] Q-|Q|^{p-1}Q=0,\quad Q\in H^2.
\end{eqnarray}
%For $\mu=0$, 
The existence of $u=Q_p$ in (\ref{1.5}) was proven for $p\in (1, \frac{2d}{(d-4)^+}-1)$  in e.g. \cite{BeFrVi14a,zhuzhangyang2010}, where   
$Q_p$ serves as a minimizer for the $J$-functional in (\ref{JpdU}).  
The uniqueness for $Q_p$
seems still open for $d> 1$ except in one dimension \cite{FrLen13a}.  
%However,  one can choose $Q_p$ to be radially symmetric in $H^2$ in the case of $p=1+2\sigma$ when $\sigma$ is an integer,  see \cite{BL2015}
We state the first main theorem concerning the existence and stability for (\ref{VP}). 
\begin{theorem}\label{th1}  Let  $p\in (1,1+\frac8d)$ and
 $\mu\in \mathbb{R}$ satisfy one of the following conditions:
\begin{itemize}
\item [(1)] $1<p<1+\frac{4}{d}$, $\forall \mu \in (0,+\infty)$;
\item [(2)] $1+\frac{4}{d}\leq p<1+\frac{8}{d}$,   $\forall \mu \in (0, \mu_0)$;
\item [(3)] $1< p<1+\frac{8}{d}$,  $\mu=0$;
 \item [(4)] $1< p<1+\frac{8}{d}$,  $\forall \mu\in (-\lambda_0,0)$ for some
 $\lambda_0:=\lambda_0(p,d, \|Q_p\|_2)>0$.
\end{itemize} 
Then the set $\mathcal{M}_{\mu}\neq \emptyset$ and is orbitally stable.
\end{theorem}
%Comparing \cite{BonhNa15} for the range of $\mu<0$ (-2\sqrt{\om}<\mu<2\sqrt{\om})
% Only the  existence rather than the stability of standing waves for (\ref{BNLS}) was shown in \cite[Theorem 1.1]{BonhNa15}  for some negative $\mu$. %$p\in  (1,1+\frac{8}{d-4})$ 

Note that when $\mu=0$, %(the case in Theorem \ref{th1} (3)) 
 the  problem \eqref{VP} reduces to the following:
\begin{eqnarray}\label{VP0}
m_0:=\inf\limits_{u\in B_1} E_0(u),
\end{eqnarray}
with
$$ E_0(u):= \frac 12 \int|\Delta  u|^{2}dx-\frac {1}{p+1} \int |u|^{p+1}dx.$$
%which is related to the classical biharmonic equation
From Theorem \ref{th1} (3) and Lemma \ref{lm3.4}, we  have following corollary.
\begin{corollary}\label{col1}
Let $p\in(1,1+\frac{8}{d})$ and $\mu=0$. Then $-\iy<m_0<0$ and (\ref{VP0}) admits a minimizer $u_0\in B_1$. Further there exists a Lagrange multiplier $\omega_0 \in \mathbb{R}$, such that $(u_0, \omega_0)$ solves \eqref{1.3} with $\mu=0$.   %\begin{equation}\label{equ2}
%\Delta^2 u + \omega u-|u|^{p-1}u=0,\ \ u\in H^2.  \end{equation}
\end{corollary}

%\begin{remark} Following the same proof of Theorem \ref{th1} given in Section \ref{s:Q-mu_subcri}, the statements of Theorem \ref{th1} remain true if we consider the following minimization problem \eqref{1.11} instead of \eqref{VP}.  Let  $\mu\in \mathbb{R}$.  For any prescribed mass level $c>0$, let $B_c:=\{u\in H^2\ |\ \int |u|^2 dx = c \}$ and
%\begin{eqnarray}\label{1.11}
%m_{\mu,c}:=\inf\limits_{u\in B_c} E_{\mu}(u),  \end{eqnarray}
%where $E_\mu$ is given by \eqref{E:4NLS-mu}. %Then (1)-(4) in  Theorem \ref{th1} is valid with $ \mathcal{M}_{\mu}$ replaced with
%\begin{eqnarray}\label{1.12}
 %   \mathcal{M}_{\mu,c}:=\{u\in B_c \;\vert \;   u\;\text{ is a  minimizer of (\ref{1.11})}\}.  \end{eqnarray}
%\end{remark} 

In view of Theorem \ref{th1} along with the profile decomposition argument, 
we obtain in Section \ref{s:Q-mu_subcri} certain asymptotic behavior of minimizers for $m_{\mu}$ as $\mu\to 0$.
\begin{corollary}\label{th1.2}
Let  $1<p<1+\frac{8}{d}$. Let $\{\mu_k\}_{k=1}^{+\infty}$ be a sequence with $\mu_k\to 0$ as $k\to +\infty$ and let $\{u_k\}_{k=1}^{+\infty}\subset B_1$ be a sequence of minimizers for $m_{\mu_k}$. Then there exists  $u_0\in B_1$ and a subsequence $\{u_{k_j}\}\subset \{u_k\}$
such that as $j\to+\iy$,
$$u_{k_j} \longrightarrow u_0 \quad \mbox{in } H^2.$$
In particular, $u_0\in B_1$ is a minimizer of $m_0$, where $m_0$ is given by \eqref{VP0}.
\end{corollary}

Concerning the gap in Theorem \ref{th1} (2): $1+\frac{4}{d}\leq p<1+\frac{8}{d}$ and 
$\mu \geq \mu_0>0$, we obtain the following theorem in Section \ref{s:Q-mu_subcri}.
\begin{theorem}\label{th1.1} Let  $1+\frac{4}{d}\leq p<1+\frac{8}{d}$ and $\mu_0$ be defined as in (\ref{1.7.1}), 
then for 
any $\mu \in (\mu_0, +\infty)$, $\mathcal{M}_{\mu} = \emptyset$, namely, $m_{\mu}$
has no  minimizers for (\ref{VP}).
\end{theorem}

%\begin{remark}\label{rem01}
We would like to mention that it is still unknown for us whether $\mathcal{M}_{\mu} = \emptyset$ or not in the critical case $\mu=\mu_0$, $p\in [1+\frac{4}{d},1+\frac{8}{d})$ 
and also in the case $\mu\in (-\infty, -\lam_0]$, $p\in(1,1+\frac{8}{d})$. 
{Nevertheless}, 
from Theorem \ref{th1} and Theorem \ref{th1.1} as well as their proofs, one may observe that when $1<p<1+\frac8d$, the term $\frac{\mu}{2}\|\nabla u\|_2^2$ in the 
energy functional  affects the existence of minimizers of $m_{\mu}$ in a subtle  manner. 
%\end{remark}

In Section \ref{s:L2-critic}, we consider the $L^2$-critical case $p=1+\frac8d$.  Note that in this case, the terms $\|\Delta u\|_2^2$ and $\|u\|_{p+1}^{p+1}$ of the functional $E_{\mu}(u)$  grow at the same rate;  they play competing roles in the analysis, see e.g. \eqref{00001}. It seems difficult to know which term might be more predominant. Hence we turn to study the following minimization problem:  For given
 $\mu\in \mathbb{R}$ and $b>0$, consider
\begin{equation}\label{1.7}
m_{\mu, b} := \inf_{u\in B_1}E_{\mu,b}(u),\quad\tag{VP-b}
\end{equation}
where
\begin{equation}\label{1.8}
E_{\mu,b}(u):=\frac{1}{2}\|\Delta u\|_2^2 + \frac{\mu}{2}\|\nabla u\|_2^2 - \frac{b}{2+\frac{8}{d}}\int |u|^{2+\frac{8}{d}}dx.
\end{equation}
 Denote the set of all minimizers for $m_{\mu,b}$ by
\begin{eqnarray}\label{Mub}
\mathcal{M}_{\mu,b}:=\{u\in B_1 \;\vert \;   u\;\text{ is a  minimizer of (\ref{1.7})} \}.
\end{eqnarray}
We define the orbital stability for $\mathcal{M}_{\mu,b}$ the same way as in 
Definition \ref{def:orb-stabi} per substituting $\mathcal{M}_{\mu,b}$ for $\mathcal{M}_{\mu}$.  

Let $Q^*$ be given in \eqref{1.4}.  Given $\mu\in \R$,  denote
\begin{eqnarray}\label{4.00}
b^*:=\|Q^*\|_2^{\frac{8}{d}},\quad  b_*:=b^*\Big[1+ \dfrac{\|Q^*\|_2^2}{4\|\Delta Q^*\|_2^2}
\left(\mu^2+\dfrac{4\|\nabla Q^* \|_2^2}{\|Q^*\|_2^2}\mu\right)\Big].
\end{eqnarray}
 %$0<b_*<b^*$ if $\mu\in (-\dfrac{4\|\nabla Q^*\|^2_2}{\|Q^*\|^2_2}, 0)$ see Proposition \ref{prop4.1} 
 In Section 4,  we obtain the following theorem for the set of ground states for (\ref{1.7}). 
\begin{theorem}\label{th2}  Let $p=1+\frac8d$. Suppose 
 $\mu\in (-\dfrac{4\|\nabla Q^*\|^2_2}{\|Q^*\|^2_2}, 0)$ and $b \in (b_*,b^*)$. 
Then  the set $\mathcal{M}_{\mu,b}\neq \emptyset$ and is orbitally stable.
\end{theorem}

%\begin{remark}\label{rem012} the assumptions in Theorem \ref{th2} on $\mu$ and $b$ are probably technical, whose aim is to show that any minimizing sequence of $m_{\mu,b}$ is non-vanishing in the sense of $L^{q}(\R^d)$ for all $q\in(2,\frac{2d}{(d-4)^{+}})$, see Lemma \ref{lm4.5} and Lemma \ref{lm4.6} 
The  existence of $\mathcal{M}_{\mu,b}$ is proved in Proposition \ref{prop4.1} and consequently the orbital stability follows the same way as Theorem \ref{th1}.
In view of  Lemma \ref{lm4.6} $(i)$,   $0< b_*<b^*$ if $\mu\in (-\dfrac{4\|\nabla Q^*\|_2^2}{\|Q^*\|_2^2}, 0)$.
Note that when $\mu<0$ and $b\geq b^*$,   Lemma \ref{lm4.1} shows  $m_{\mu,b}=-\infty$, meaning that  \eqref{1.7} is unsolvable. 
Also Lemma \ref{lm4.0} says when $\mu>0$,  we have $\mathcal{M}_{\mu,b}= \emptyset$ for all $b>0$. 
%also see similar result in \cite[Theorem 1.2]{BCGJ} 
The case $\mu=0$ admits ground state solutions if and only if $b=b^*$ exactly. 
%\end{remark} 
 
As a second main theorem, we  gives an alternative formulation of  Theorem \ref{th2} on $\mu<0$.
Define $\lam_1:=\dfrac{4\|\nabla Q^*\|^2_2}{\|Q^*\|^2_2}$ and\\
 \begin{align*}
 \beta:= \Big[1+ \dfrac{\|Q^*\|_2^2}{4\|\Delta Q^*\|_2^2}
\left(\mu^2+\dfrac{4\|\nabla Q^* \|_2^2}{\|Q^*\|_2^2}\mu\right)\Big]^{\frac{d}{8}}.  %  $b=1$ in \eqref{1.8} 
\end{align*}
Then $\beta\in (0,1)$ if $\mu\in  (-\lam_1,0)$.
\begin{theorem}\label{th22} Let $p=1+\frac8d$. 
For given $\mu\in (-\lam_1, 0)$, if $\beta \norm{Q^*}_2<c< \|Q^*\|_2$, then the minimization problem 
\begin{equation}\label{MinP}
m_{\mu,1}^c:=\inf\limits_{u\in B_c} E_{\mu,1}(u) 
\end{equation}
has  ground state solutions, 
where $E_{\mu,1}(u)=E_{\mu}(u)$ is defined as  in (\ref{E:4NLS-mu})
%\[ E_{\mu,1}(u):= \frac 12 \int|\Delta  u|^{2}dx+\frac{\mu}{2} \int|\nabla u|^{2}dx -\frac {1}{2+\frac{8}{d}} \int |u|^{2+\frac{8}{d}}dx \]
and  the level set $B_c:=\{u\in H^2\ |\ \int |u|^2 dx = c^2\}$. 
 In addition, if we denote by
\begin{equation}\label{e:minPc}
    \mathcal{M}_{\mu}^c:=\{u\in B_c \, |\ u\;\text{ is a  minimizer of (\ref{MinP})}  \}
\end{equation}
 the set of  ground state solutions in $B_c$, then $\mathcal{M}_{\mu}^c\neq \emptyset$ and is orbitally stable under the flow of (\ref{BNLS})-(\ref{BNLS1}).
\end{theorem}
Theorem \ref{th22} can be proved nearly verbatim following the same proof for Theorem \ref{th2},
which supplements the results  in Theorem 1.2 and Theorem 1.4 in \cite{BCGJ} on the case $\mu>0$. 
%\begin{remark}\label{rem023}
 Theorem \ref{th1} and Theorem \ref{th22}
  show that the sign of the second-order dispersion has crucial effect on the construction  of orbitally stable standing waves for the BNLS 
  especially when $\mu$ is negative.  This is the case where $\De^2$ and $-\mu\De$  have played opposite roles for the dispersion of the energy that arises in physics
   \cite{BaruchFibichMandelbaum2010,BL2015,Karpman1996,Karpman2000}. %the two dispersive terms in the energy having different signs
    Notably, in the mass critical case $p=1+\frac8d$,  we find that when $-\lam_1<\mu<0$, the term $-\mu\De$ contributes to the  
     existence of orbitally stable ground states for (\ref{BNLS}),  %on the mass level between $\beta \norm{Q^*}_2$ and $\norm{Q^*}_2$
      while in the case $\mu>0$ there exists no ground states. 
      Note that the result on (\ref{BNLS}) for $\mu=0$ corresponds to the classical second order NLS, both accounting for the $L^2$-critical regimes, 
      cf.  \cite{Fibich2002,Weinstein1983,zhuzhangyang2010}. 
      
 We would like to mention that  our proofs of the main theorems 
   give a simple systematic method to show the existence of g.s.s. 
    that include non-radial solutions for  (\ref{BNLS}) based on  the
 profile decomposition  analysis. %introduced in \cite{zhuzhangyang2010}. 
 From Theorem \ref{th1} and Theorem \ref{th1.1} we see the upper bound $\mu_0$ is sharp for $\mu>0$ and $p\in [1+\frac4d,1+\frac{8}{d})$.
 Moreover, the variational argument allows to determine a lower bound of  $\mu<0$ for $p\in (1, 1+\frac8d]$ 
 regarding  the existence of ground states for (\ref{VP}) and (\ref{1.7}). 
 From the proofs we conjecture that the lower bounds $\lam_0$ and $\lam_1$ in Theorem \ref{th1} and Theorem \ref{th22} are optimal,
which are intrinsically dependent on the ground state of (\ref{1.3}) with $\mu=0$. 
 However, the uniqueness and symmetry problem seems to remain unsettled other than knowing $\forall \theta, y\in\R$,
 $e^{i\theta}Q(\cdot-y)\in \mathcal{M}_\mu$ for all $Q\in \mathcal{M}_\mu$. 
 In this respect,  the papers \cite{BCMN,BCGJ} use the classical concentration-compactness method to study the existence of radially symmetric g.s.s.,
  however, the argument does not seem to directly apply to the focusing case $\mu<0$.  

For $p>1+\frac4d$, orbital stability of the  classical second-order NLS were considered in  \cite{Weinstein1986},
and later the result on NLS were significantly extended in \cite{GrillakisShatahStrauss1987}   for general Hamiltonian systems that are
invariant under a group of transformations. 
The analogous results  for (\ref{BNLS})   in the $L^2$-supercritical case 
  were studied in  \cite{BCMN,BCGJ,NataliPastor2015} via Lyapunov functional method for $\mu\ge 0$ and $p>1+\frac8d$.    
  %   \cite{CazenaveLions1982}  using the concentration compactness 
The profile decomposition method  has potential applications to the study of such problem in the case  $\mu<0$ and $p>1+\frac8d$.   
The analysis in this paper can be further extended to address  the orbital stability for higher-order Schr\"{o}dinger type equations with potentials, based on 
the analogues for the NLS  \cite{FukuizumiOhta2003,Sulem1999,zhang2000,ZhangZhengZhu2019}  and related  dynamical properties for  general  Hamiltonian partial differential equations near the standing waves  \cite{Lin2017,zhangzhu2017,zhu2016}.

The remaining of the paper is organized as follows. In Section 2, we mainly state  the local well-posedness of (\ref{BNLS})-(\ref{BNLS1}), 
 the profile decomposition  in $H^2$ and a sharp Gagliardo-Nirenberg inequality for $\De^2$. 
In Section 3 and Section 4, we shall prove Theorem \ref{th1} and Theorem \ref{th2} respectively concerning the construction of ground states for (\ref{BNLS}). %\end{document}

\renewcommand{\theequation}{\thesection.\arabic{equation}}
 \setcounter{equation}{0}
 \section{Preliminaries}\label{s:prelimin}

Throughout this paper,   we write $\int h(x)dx$ to represent the Lebesgue integral of $h$ over $\R^d$.
Let $L^p:= L^p (\R^d )$, $1\leq p\leq+\infty$ be the usual Lebesgue space equipped with the standard norm
$\|\cdot\|_{p}$. Let  $H^2(\R^d)$ denote the Sobolev space equipped with
the  norm  $\| v\|_{H^2}:=\left(\|\Delta v\|_2^2+\|v\|_2^2\right)^{1/2}$.
Let $C$ denote a positive constant that may vary from one context to another. 
For  Cauchy problem (\ref{BNLS})-(\ref{BNLS1}),  
there holds  the local well-posedness which was obtained in  \cite{BenKochSaut2000,KenigPonceVega1991, Pausader2007DPDE}. 
\begin{proposition}\label{local} Let  $\psi_{0}\in H^2$ and $1<p<\frac{2d}{(d-4)^+}-1$.
 There exists a unique solution $\psi(t,x)$
of  Cauchy problem (\ref{BNLS})-(\ref{BNLS1}) on the maximal time
interval $[0,T)$ such that $\psi\in C([0,T);H^{2})$.
There holds the blowup alternative, namely, either
$T=+\infty$ (global existence),
{or}  $0<T<+\infty$ and
$\lim\limits_{t\to T} \| \psi(t,\cdot)\|_{H^2} =+\infty$ (blow-up).
Furthermore, for all $t\in [0,T)$, $\psi$ satisfies the
following conservation laws:
\begin{itemize}
  \item [(i)] Conservation of mass: $M(\psi)=\|\psi(t,\cdot)\|^2_2=\|\psi_0\|^2_2\,$
  \item [(ii)] Conservation  of energy: 
\[E_{\mu}(\psi)= \frac 12 \int|\Delta  \psi|^{2}dx+\frac{\mu}{2} \int|\nabla \psi|^{2}dx
 -\frac {1}{p+1} \int |\psi|^{p+1}dx=E_{\mu}(\psi_{0}).\]
\end{itemize}
\end{proposition}

For $p\in (1, 1+\frac8d]$ we shall use the profile decomposition in $H^2$
as a main tool to show the existence of 
g.s.s. for (\ref{BNLS}), or equivalently for (\ref{VP}) and (\ref{1.7}).
The exponent $p=1+\frac8d$ is mass-critical. Heuristically  
this can be observed from the scaling invariance $u\mapsto u_\lam:= \lam^{\frac{4}{p-1}} u(\lam^4 t, \lam x) $ which preserves the $L^2$-norm if and only if 
$p=1+\frac8d$. The following proposition was obtained  in \cite{zhuzhangyang2010}. 

\begin{proposition}\label{decomposition} %(\cite{zhuzhangyang2010})
Let $\{v_n\}_{n=1}^{+\infty}$ be a bounded sequence in $H^2$. Then
there exist a subsequence of $\{v_n\}_{n=1}^{+\infty}$ \textup{(}still denoted
$\{v_n\}_{n=1}^{+\infty}$\textup{)}, a family $\{x_n^j\}_{j=1}^{+\infty}$ of
sequences in $\mathbb{R}^d$ and a sequence $\{V^j\}_{j=1}^{+\infty}$
of $H^2$ functions such that
\begin{itemize}
  \item [(i)] for all $k\neq j$,  $|x_n^k-x_n^j|\rightarrow+\infty \quad {as}\; n\rightarrow +\infty$,
  \item [(ii)] for all $l\geq 1$,
   \be\label{2.5}v_n=\sum\limits_{j=1}^{l}V^j(\cdot-x_n^j)+r_n^l\ee
   with
   \be\label{2.6}\limsup_{n\rightarrow +\infty}\|r_n^l\|_q\rightarrow 0  \ \ {\rm
   as}\ \
  l\rightarrow +\infty,\ee
  for each $q\in(2,\frac{2d}{(d-4)^{+}})$.
\end{itemize}
  Moreover, we have, as $n\rightarrow +\infty$
\begin{align}
\label{2.7}
\| v_n\|_2^2=&\sum\limits_{j=1}^l\| V^j\|_2^2+\|
  r_n^l\|_2^2+o_n(1)\\
 \label{2.08}\| \nabla v_n\|_2^2=&\sum\limits_{j=1}^l\| \nabla V^j\|_2^2+\|
\nabla r_n^l\|_2^2+o_n(1)\\
  \label{2.8}\| \Delta v_n\|_2^2=&\sum\limits_{j=1}^l\| \Delta V^j\|_2^2+\|
 \Delta r_n^l\|_2^2+o_n(1)\\
 \label{2.37}\|\sum_{j=1}^{l}V^{j}(x-x_n^j)\|_{q}^{q}=&\sum_{j=1}^{l}\|V^{j}(x-x_n^j)\|_{q}^{q}+o_n(1),
 \end{align}
 where $o_n(1)\rightarrow 0$ as $n\rightarrow +\infty$.
 \end{proposition}
 A primary  advantage of the profile decomposition is  the %  functionals. 
   almost orthogonality that can be used to defeat the lack of compactness of a given bounded sequence, 
 as can be seen from the proof of Theorem \ref{th1} in Section 3. 
In solving the variational problem \eqref{VP}, we also need the following sharp
Gagliardo-Nirenberg type inequality obtained %Zhu, Zhang   Yang
in  e.g. %Fibich, Ilan and Papani
\cite{Fibich2002} %(subcritical case) 
and \cite{zhuzhangyang2010}. %(critical case, supercritical case).
\begin{proposition}\label{prop2.4} Let $1<p<\frac{2d}{(d-4)^+}-1$. Then for all $v\in H^2$
 \be\label{GN2}
 \|v\|_{{p+1}}^{p+1}\le  B_{p,d}
 \| v\|_{2}^{\frac{(4-d)p+4+d}{4}}\|\Delta v\|_{2}^{\frac{(p-1)d}{4}},\ee
where $B_{p,d}=\frac{p+1}{2 \| Q_p\|_{2}^{p-1}}$ and 
$Q_p$ is a ground state solution of  \eqref{1.5}.%\edz{Please check the constant $B_{p,d}$ in the initial draft;what I obtain is (\ref{BR_pd})}
\end{proposition}
\begin{remark}\label{rek2.5}
Define the $J$-functional in $H^2$ for $p\in (1, 1+\frac{8}{(d-4)^+})$   
\begin{align}
&J_{p,d}(u)= \frac{\norm{\De u}_2^{\frac{(p-1)d}{4}} \norm{ u}_2^{p+1-\frac{(p-1)d}{4}} }{ \norm{ u}_{p+1}^{p+1}} \, .\label{JpdU}
\end{align}
Then $B_{p,d}\inv =\min_{0\ne u\in H^2} J_{p,d}(u)$.  %where $ B_{p,d}$ is the optimal constant satisfying 
%\begin{align*}  & \norm{ u}_{p+1}^{p+1} \le B_{p,d} \norm{ u}_2^{p+1-(p-1) d/4} \norm{\De u}_2^{(p-1) d/4} .%\label{optimal_B}
%\end{align*}
  The paper \cite{zhuzhangyang2010} shows that minimizers for $J_{p,d}$ exist, which are exactly
given by the ground state solutions that verifies  \eqref{1.5},  in virtue of the Pohozaev identities in Proposition \ref{pohozaev4NLS} with $\mu=0$.
%\begin{align}
%let $v^t (x)=t^{d/2}v(tx), t>0$ then 
%E_\mu(v^t)=\frac{t^4}{2}\norm{\De v}_2^2+\frac{\mu t^2}{2}\norm{\nabla v}_2^2
%+\frac{t^{\frac{(p-1)d}{2}}}{p+1}\norm{ v}_{p+1}^{p+1}\\
%%&P(u)=\frac{\ga(d-4) }{2} \int |\De u|^2+\frac{\mu(d-2)}{2}\int |\nabla u|^2+\frac{\om d }{2}\int |u|^2-\frac{cd}{p+1}\int |u|^{p+1}\\
%&(1-\frac{(p-1)d}{4(p+1)}) \int |u|^{p+1} dx= \omega \int | u|^2dx\label{e:up-om-u2}\\
%& \int |\De u|^2 dx=  \frac{ (p-1)d}{4(p+1)} \int |u|^{p+1}dx\,\label{e:DeU-up} \\
%& \int |\De u|^2 dx+\frac{\mu}{2}\int |\nabla u|^2= \frac{ d(p-1)}{4(p+1)} \int |u|^{p+1}dx\,,\label{pohozaev:4NLS-mu}   \end{align}
Moreover, we have 
\begin{align}\label{B-L2critic}
&B_{p,d}\inv = \frac{2}{p+1}\norm{Q_p}_2^{p-1} \,.
\end{align}
%which shows $\norm{Q_p}_2$ is independent of the choices of any particular $J$-minimizer. 
It is worth mentioning  that owing to Pohozaev identities and the scaling-invariance of the $J$-functional, 
any  minimizer of the $J$-functional is a solution of \eqref{1.5} up to a scaling $u(x)=\beta Q(\al x )$ 
 and vise versa, where $Q$ is the ground state solution of (\ref{1.5}). 
%\begin{align*} &\De^2+R-|R|^{p-1}R=0 \end{align*}
  Also, any $J$-minimizer is an $E_0$-minimizer modular scaling and vise versa. 
\end{remark}
The Pohozaev type identity has a general version as given in  \cite[Appendix 4.10]{Gou17t}, 
cf.  \cite{BCGJ,BL2015,Fibich2002} for some special cases. %as we occasionally need in  later 
\begin{proposition}\label{pohozaev4NLS} 
Let $p\in (1, 1+\frac{8}{(d-4)^+}) $.  Let $u\in H^2$ be a weak solution of %\eqref{4abc}%DeU-ga-al-p}
\begin{align}\label{4abc}
\gamma (\De)^2 u -\mu\De u -c |u|^{p-1}u=-\omega u\,, \qquad %(\om=\al)
\end{align}
where $\gamma, \mu, c, \omega$ are constant coefficients. Then we have %$I(u)=P(u)=Q(u)=0$, where 
\begin{align}
%&I(u)=
&\gamma \int |\De u|^2+\mu \int |\nabla u|^2+\omega \int |u|^2-c\int |u|^{p+1}=0 \notag \\
%&P(u)=\frac{\ga(d-4) }{2} \int |\De u|^2+\frac{\mu(d-2)}{2}\int |\nabla u|^2+\frac{\omega d }{2}\int |u|^2
%-\frac{cd}{p+1}\int |u|^{p+1}\label{e:Pu}\\
%&Q(u)= 
&\ga \int |\De u|^2+\frac{\mu}{2} \int |\nabla u|^2-\frac{ c(p-1)d}{4p+4} \int |u|^{p+1}=0. \label{e:Qu}%\\
%=&\frac{d}{4} I(u)-\frac12 P(u).\notag
\end{align}
\end{proposition}%\edz{all three checked }
If $p=1+\frac8d$, $\ga=1$, $c=1$ and $\omega=\frac4d$, then from (\ref{1.4}) we have with $u=Q^*=Q_{1+8/d}$
\begin{align}
& \int |Q^*|^2=\int |\De Q^*|^2=\frac{1}{1+\frac4d} \int |Q^*|^{2+\frac8d}\label{Q-De-p}\\
%& \int |Q^*|^2=\int |\De Q^*|^{2} 
%&E_{0}(Q^*)=\frac12\int |\De Q^*|^2-\frac{1}{p+1} \int |Q^*|^{p+1}=\frac{(p-1)d-8}{8(p+1)} \int |Q^*|^{p+1}=0. \quad checked\\
%&E_{0,b}(Q^*)=\frac12\int |\De Q^*|^2-\frac{b}{p+1} \int |Q^*|^{p+1}=\frac{(p-1)d-8b}{8(p+1)} \int |Q^*|^{p+1} \quad checked\\
&E_{0,b}(Q^*)=\frac{1-b}{2+\frac8d} \int |Q^*|^{2+\frac8d}=\frac{1-b}{2} \int |Q^*|^2.\label{E0b:Q} %\qquad checked
\end{align}

To conclude the preliminary section, 
we recall the definition of the orbital stability for (\ref{BNLS}), which is stated in Theorems \ref{th1}, \ref{th2} and \ref{th22}.
\begin{definition}\label{def:orb-stabi}
		The set $\mathcal{M}_{\mu}$ is said to be {\bf orbitally stable} if any given $\varepsilon>0$, there exists $\delta>0$ such that for any initial data $\psi_0$ satisfying
		\[\inf_{u \in \mathcal{M}_{\mu}} \|\psi_0 -u\|_{H^2} <\delta,\]
		the corresponding solution $\psi(t,x)$ of the Cauchy problem (\ref{BNLS})-(\ref{BNLS1})  satisfies \[ \inf_{u \in \mathcal{M}_{\mu}} \|\psi(t,\cdot)-u\|_{H^2}<\varepsilon\qquad  \text{for  all  $t$.}
\]
 In other words, if the initial data $\psi_0$ is close to an
orbit $u\in \mathcal{M}_{\mu}$, then the corresponding solution $\psi(t,x)$ of the system
(\ref{BNLS})-(\ref{BNLS1}) remains close to the set of orbits $ \mathcal{M}_{\mu}$ for all time. The analogous definition applies to the orbital stability for 
$M_{\mu,b}$ and $M_\mu^c$ in (\ref{Mub}) and (\ref{MinP}).
\end{definition}

\renewcommand{\theequation}{\thesection.\arabic{equation}}
 \setcounter{equation}{0}
\section{Construction of standing waves in the $L^2$ subcritical case}\label{s:Q-mu_subcri}
Let $\mu\in \mathbb{R}$ and $1<p<1+\frac{8}{d}$, then we note that the variational problem (\ref{VP})
\[m_{\mu}=\inf\limits_{v\in B_1} E_{\mu}(v)\]
is well-defined, namely $m_{\mu}\neq -\infty$.
%By applying the H\"{o}lder interpolation estimate and the
%Gagliardo-Nirenberg  (\ref{GN2}), we have the
%following estimate for the functional $E_{\mu}(u)$
%\be\label{3.2}\begin{array}{lll} \vspace{0.3cm}
% E_{\mu}(v)&\geq \frac12 \|\Delta v\|_2^2+\frac{\mu}{2}\|\nabla
%v\|_2^2- C \|v\|_2^{\frac{(4-d)p+4+d}{4}}\|\Delta
%v\|_2^{\frac{(p-1)d}{4}}\\
%&\geq
%\frac12 \|\Delta v\|_2^2-\varepsilon \|\Delta v\|_2^2 -C(\varepsilon, \mu,p,d,
%\|v\|_2)\end{array}\ee for any $0<\varepsilon<\frac12$. This implies
%that under the restriction: $\|v\|_2^2=M$, $E_{\mu}(v)$ has a lower bound, $E_{\mu}(v)\geq -C(\varepsilon,\mu,p,d, M)$.
Indeed, by the Gagliardo-Nirenberg inequality (\ref{GN2}), we have
\begin{eqnarray}\label{00001}
E_{\mu}(v)\geq \frac12 \|\Delta v\|_2^2+\frac{\mu}{2}\|\nabla
v\|_2^2- C \|v\|_2^{\frac{(4-d)p+4+d}{4}}\|\Delta
v\|_2^{\frac{(p-1)d}{4}}, \quad \forall v\in H^2,
\end{eqnarray}
where $C=C(p,d,\|Q_p\|_2)>0$.  When $\mu\geq 0$, \eqref{00001} implies that
\begin{eqnarray}\label{00002}
E_{\mu}(v)\geq \frac12 \|\Delta v\|_2^2- C \|v\|_2^{\frac{(4-d)p+4+d}{4}}\|\Delta
v\|_2^{\frac{(p-1)d}{4}}, \quad \forall v\in H^2.
\end{eqnarray}
When $\mu< 0$, by the inequality $\|\nabla v\|_2^2\leq \|\Delta v\|_2\|v\|_2$, \eqref{00001} implies that
\begin{eqnarray}\label{00003}
E_{\mu}(v)\geq \frac12 \|\Delta v\|_2^2+ \frac{\mu}{2}\|\Delta v\|_2\|v\|_2 - C \|v\|_2^{\frac{(4-d)p+4+d}{4}}\|\Delta
v\|_2^{\frac{(p-1)d}{4}}, \quad \forall v\in H^2.
\end{eqnarray}
Noting that $0<\frac{(p-1)d}{4}<2$ as $1<p<1+\frac{8}{d}$, we conclude from \eqref{00002} and \eqref{00003} that $m_{\mu}\neq -\infty$, then  (\ref{VP}) is well-defined.

Before solving  (\ref{VP}), it is necessary to study the properties of $m_{\mu}$, which  we shall 
prove in Lemmas \ref{3.1} to \ref{lm3.4}. 
\begin{lemma}\label{lm3.1} Let  $1<p<1+\frac{8}{d}$, then
\begin{itemize}
  \item [(a)] $m_{\mu}$ is non-decreasing with respect to $\mu\in \mathbb{R}$;
  \item [(b)] $m_{\mu}$ is continuous at each $\mu\in \mathbb{R}$.
\end{itemize}
\end{lemma}

\begin{proof} To prove (a), 
we observe that for any $\mu_1, \mu_2\in \mathbb{R}$ with $\mu_1 < \mu_2$, there holds 
$$E_{\mu_1}(u)<E_{\mu_2}(u),\quad \forall u\in B_1.$$
Then by the definition of $m_{\mu}$, we have $m_{\mu_1}\leq m_{\mu_2}$, thus $(a)$ is proved.

 For $(b)$, we first show that for any $\mu_n \to \mu^-$ as $n\to +\infty$, $m_{\mu_n}\to m_{\mu}$. Indeed, for each $n\in \mathbb{N}$, by the definition of $m_{\mu_n}$, there exists a $u_n\in B_1$ such that
\begin{eqnarray}\label{3.1}
m_{\mu_n} \leq E_{\mu_n}(u_n)<m_{\mu_n}+ \frac{1}{n} < m_{\mu}+ \frac{1}{n}\,.
\end{eqnarray}
Then by \eqref{00001} and the inequality $\|\nabla u_n\|_2^2\leq \|\Delta u_n\|_2\|u_n\|_2$, we see that $\{u_n\}_{n=1}^{\infty}$ is bounded in $H^2$. Thus from \eqref{3.1},
\begin{eqnarray*}
m_{\mu} \leq E_{\mu}(u_n)&=&E_{\mu_n}(u_n) + (\mu-\mu_n) \dfrac{\|\nabla u_n\|_2^2}{2}\\
&<& m_{\mu}+ (\mu-\mu_n) \dfrac{\|\nabla u_n\|_2^2}{2}   + \frac{1}{n}\,,
\end{eqnarray*}
by which we conclude that $m_{\mu_n}\to m_{\mu}$ as $\mu_n \to \mu^-$. Similarly, we can show that $m_{\mu_n}\to m_{\mu}$ as $\mu_n \to \mu^+ $.  This  proves the continuity of $m_{\mu}$ for all $\mu\in \mathbb{R}$.
\end{proof}

%Concerning the value of $m_{\mu}$ with $\mu\in \mathbb{R}
\begin{lemma}\label{lm3.0} Let $p>1$, then $m_{\mu}\leq 0$ for any $\mu\in \mathbb{R}$.
\end{lemma}
\begin{proof} Let $v_0\in B_1$ be fixed and define the scaling for all $\rho>0$
\begin{align}\label{v0-rho}
v^{\rho}(x):=\rho^{\frac d2} v_0(\rho x).
\end{align}
 Then $v^{\rho} \in B_1$ for any $\rho>0$,  and
\begin{eqnarray}\label{3.3.1}
E_{\mu}( v^{\rho} )=\frac{\rho^4}{2} \|\Delta v_0\|_2^2+\frac{\mu \rho^2}{2}\|\nabla
v_0\|_2^2-\frac{\rho^{\frac{(p-1)d}{2}}}{p+1}\|v_0\|_{p+1}^{p+1}.
\end{eqnarray}
Thus for any $\mu\in \mathbb{R}$, by \eqref{3.3.1} and the definition of $m_{\mu}$, we have $m_{\mu}\leq \lim\limits_{\rho\to 0^+}E_{\mu}(v^{\rho})=0.$ Then the lemma is proved.
\end{proof}

%More precisely, we shall prove the following.

\begin{lemma}\label{lm3.2}
Let $1<p<1+\frac{8}{d}$ and $\mu >0$. Then the following properties for $m_\mu$ hold.
\begin{itemize}
  \item [(1)] If $1<p<1+\frac{4}{d}$, $m_{\mu}<0$ for any $\mu>0$.
  \item [(2)] If $1+\frac{4}{d}\leq p<1+\frac{8}{d}$, let
             \begin{eqnarray}\label{3.2}
             \mu_0:=\sup \{\mu>0\ | \ m_{\mu}<0 \}.
             \end{eqnarray}
  Then $0<\mu_0<+\infty$ and
$$\left\{
\begin{array}{l}
m_{\mu} < 0, \quad 0<\mu <\mu_0 ,\\
m_{\mu} = 0, \quad \mu \geq\mu_0.
\end{array}
\right.$$

\end{itemize}
\end{lemma}

\begin{proof} Given $\mu>0$, from (\ref{v0-rho}) and \eqref{3.3.1} we have
\begin{eqnarray}\label{3.3.2}
\dfrac{E_{\mu}( v^{\rho} )}{\rho^{\frac{(p-1)d}{2}}}=\frac{\rho^{\frac{8+d-pd}{2}}}{2} \|\Delta v_0\|_2^2+\frac{\mu \rho^{\frac{4+d-pd}{2}}}{2}\|\nabla
v_0\|_2^2-\frac{1}{p+1}\|v_0\|_{p+1}^{p+1}.
\end{eqnarray} 
If $1<p<1+\frac{4}{d}$,  then  $\frac{4+d-pd}{2}>0$ and it follows from \eqref{3.3.2}
that  there exists some $\rho_0>0$ sufficiently small such that $m_{\mu}\leq E_{\mu}(v^{\rho_0})<0$.
 Thus $(1)$ is verified.

If $1+\frac{4}{d}\leq p<1+\frac{8}{d}$,  we let $\rho=\sqrt{\mu}$. 
Then \eqref{3.3.2} implies that
\begin{eqnarray}\label{3.3.3}
\dfrac{E_{\mu}( v^{\rho} )}{\mu^{\frac{(p-1)d}{4}}}=\frac{\mu^{\frac{8+d-pd}{4}}}{2} \|\Delta v_0\|_2^2+\frac{ \mu^{\frac{8+d-pd}{4}}}{2}\|\nabla
v_0\|_2^2-\frac{1}{p+1}\|v_0\|_{p+1}^{p+1}\,,
\end{eqnarray}
from which we conclude that $m_{\mu}<0$ if $\mu>0$ is small enough. Thus by the definition of $\mu_0$, we must have $\mu_0>0$. 
To prove  $\mu_0<\infty$, it suffices to show  $m_{\mu}=0$ 
for  $\mu>0$ large enough.

For this purpose, we recall that in \cite[(2.4)]{BCMN} the authors established the following estimate:
\begin{eqnarray}
\int |v|^{p+1}dx\leq C_{p,d} \|v\|_2^{p-1} \|\Delta v\|_2^{\frac{pd-d-4}{2}}\|\nabla v\|_2^{\frac{8+d-pd}{2}},\qquad \forall v\in H^2,
\end{eqnarray}
for some constant $C_{p,d}>0$. Thus for all $u\in B_1$,
\begin{eqnarray}\label{3.3.5}
E_{\mu}(u)\geq \dfrac{1}{2}\|\Delta u\|_2^2 + \dfrac{\mu}{2}\|\nabla u\|_2^2 - \dfrac{C_{p,d}}{p+1} \|\Delta u\|_2^{\frac{pd-d-4}{2}}\|\nabla u\|_2^{\frac{8+d-pd}{2}}.
\end{eqnarray}
If $p=1+\frac{4}{d}$,  we have  by \eqref{3.3.5},
\begin{eqnarray}\label{3.3.6}
E_{\mu}(u)\geq \Big[\dfrac{\mu}{2}- \frac{C_{p,d}}{p+1}\Big]\|\nabla u\|_2^2 \qquad \forall u\in B_1.
\end{eqnarray}
In view of Lemma \ref{lm3.0} and \eqref{3.3.6}, 
we must have  $m_{\mu}=0$ whenever $\mu>0$ large enough. 
If  $1+\frac{4}{d}< p<1+\frac{8}{d}$, we have by the
Young inequality, for any $\varepsilon>0$,
$$
\|\Delta u\|_2^{\frac{pd-d-4}{2}}\|\nabla u\|_2^{\frac{8+d-pd}{2}}\leq \varepsilon \|\Delta u\|_2^{\frac{pd-d-4}{2}p'} + C(\varepsilon)\|\nabla u\|_2^{\frac{8+d-pd}{2}q'},
$$
where $\frac{1}{p'}+\frac{1}{q'}=1$ and $C(\varepsilon)=(\varepsilon p')^{-q'/p'}(q')^{-1}$.
Let $\varepsilon$, $p'$ be such that
$$
\begin{cases}
\frac{C_{p,d}}{p+1}\varepsilon =\frac{1}{2}\\
\frac{pd-d-4}{2}p'=2
\end{cases}
\Longleftrightarrow  \quad
\left\{
\begin{array}{l}
\varepsilon = \frac{p+1}{2C_{p,d}}\\
p' = \frac{4}{pd-d-4}\,.
\end{array}
\right.
$$
Then $q'=\frac{4}{8+d-pd}$ and $\frac{8+d-pd}{2}q'=2$. So, by \eqref{3.3.5} we obtain the following estimate similar to (\ref{3.3.6}):
\begin{eqnarray*}
E_{\mu}(u)\geq \Big[\dfrac{\mu}{2}- \frac{C_{p,d}}{p+1}C(\varepsilon)\Big]\|\nabla u\|_2^2 \qquad \forall u\in B_1,
\end{eqnarray*}
from which we must also have  $m_{\mu}=0$ if $\mu>0$ large enough. 
Therefore, we have proved that $0<\mu_0<+\infty$. 

Finally, from the definition of $\mu_0$ and the continuity and non-decreasing monotonicity of $m_{\mu}$ in Lemma \ref{lm3.1}, we conclude that $m_{\mu}<0$ if $0<\mu<\mu_0$ and $m_{\mu}=0$ if $\mu\ge \mu_0$. This completes  the proof. 
\end{proof}

Concerning the case $\mu \leq 0$, we have the following lemma.
\begin{lemma}\label{lm3.4}
Let $1<p<1+\frac{8}{d}$. Then $m_{\mu}<0$ for all $\mu \leq 0$.
\end{lemma}

\begin{proof}
Indeed, for any $\mu \leq 0$, we let $v_0\in B_1$ be fixed and consider the scaling
$v^{\rho}=\rho^{\frac d2} v_0(\rho x)$, where $\rho>0$ is an arbitrary constant. Then $v^{\rho} \in B_1$ for any $\rho>0$,  and since $\mu\leq 0$ we have
\begin{eqnarray}\label{3.4.1}
E_{\mu}( v^{\rho} )&=&\frac{\rho^4}{2} \|\Delta v_0\|_2^2+\frac{\mu \rho^2}{2}\|\nabla
v_0\|_2^2-\frac{\rho^{\frac{(p-1)d}{2}}}{p+1}\|v_0\|_{p+1}^{p+1}\nonumber\\
&\leq& \frac{\rho^4}{2} \|\Delta v_0\|_2^2-\frac{\rho^{\frac{(p-1)d}{2}}}{p+1}\|v_0\|_{p+1}^{p+1}.
\end{eqnarray}
Note that $0<\frac{(p-1)d}{2}<4$ as $1<p<1+\frac{8}{d}$, then from \eqref{3.4.1} we deduce that there exist constants $\rho_0>0$ which depends only on the values of $p,d,v_0$, such that $E_{\mu}(v^{\rho_0})<0$. Then $m_{\mu}<0$.
\end{proof}

We are ready to prove the existence of a minimizer for 
(\ref{VP}).
We show that the the infimum of (\ref{VP}) can be achieved by using the profile decomposition of bounded
sequences in $H^2$. 
\begin{proposition}\label{M} Let $1<p<1+\frac8d$.  Suppose that $\mu\in \R$ and $m_{\mu}$ satisfy one of the following conditions:
\begin{itemize}
\item [(i)] $\mu\geq 0$ and $m_{\mu}<0$;
\item [(ii)] $-\lambda_0< \mu<0$ for some $\lambda_0:=\lambda_0(p,d,\|Q_p\|_2)>0$, where $Q_p$ is given in \eqref{1.5}.
\end{itemize}
Then any minimizing sequence of $m_{\mu}$ is pre-compact in $H^2$. Moreover, there exists some $u\in B_1$  such that
 \be\label{3.3} m_{\mu}=E_{\mu}(u),\ee
namely $\mathcal{M}_{\mu}\neq \emptyset$.
\end{proposition}

\begin{proof} {\bf Case (i): } $\mu\geq 0$ and $m_{\mu}<0$.  % we claim that there exists a constant $C_0>0$, independent of $M>0$, such that
   Let $\{v_n\}_{n=1}^{\infty}\subset B_1$ be an arbitrary sequence satisfying
   \be\label{3.4} E_{\mu}(v_n)\rightarrow m_{\mu}\ \  \
{\rm as} \ \ n\rightarrow +\infty.\ee
Then for $n$ large enough, we have
\be\label{3.6}
m_{\mu}\le E_{\mu}(v_n)<\frac{m_{\mu}}{2}<0
\ee and 
\be\label{001}
\frac{1}{p+1}\int |v_n|^{p+1}dx=\frac12\|\Delta v_n\|_2^2+\frac{\mu}{2}\|\nabla v_n\|_2^2-E_{\mu}(v_n)
\geq -\frac{m_{\mu}}{2}>0.
\ee
  By H\"older inequality, we see that $\{v_n\}_{n=1}^{+\infty}\subset B_1$ is non-vanishing in $L^q$ for all $q\in(2,\frac{2d}{(d-4)^{+}})$. In addition, by \eqref{00002} and \eqref{3.4}, we see  $\{v_n\}_{n=1}^{+\infty}$ is bounded in $H^2$ if $1<p<1+\frac8d$.

Then by Proposition \ref{decomposition}, there is a weakly convergent subsequence of $\{v_n\}$ (still denoted by $\{v_n\}$) 
such that  \be\label{3.7}v_n=\sum\limits_{j=1}^lV^j(\cdot-x_n^j)+r_n^l \ee
with 
$\lim\limits_{l\rightarrow+\infty}\limsup\limits_{n\rightarrow \infty}\|r_n^l\|_q=0$  for $q\in(2,\frac{2d}{(d-4)^{+}})$, and moreover, as $n\rightarrow +\infty$,  
(\ref{2.7})-(\ref{2.37}) are ture.
Subsituting (\ref{3.7})  into the energy functional yields
\be\label{3.12}E_{\mu}(v_n)=\sum\limits_{j=1}^{l}E_{\mu}(V^j(x-x_n^j))+E_{\mu}(r_n^l)+o_{n,l}(1),\ee
where  $\lim\limits_{l\rightarrow +\infty}\lim\limits_{n\rightarrow +\infty}o_{n,l}(1)=0$.
Since $\{v_n\}_{n=1}^{+\infty}$ is non-vanishing in $L^q$ for all $q\in(2,\frac{2d}{(d-4)^{+}})$,  
we may assume  $\|V^j(x-x_n^j)\|_2\ne 0,\  \forall 1\leq j\leq l$ for all $\ell\ge 1$ without loss of generality, 
according to Lions' vanishing lemma  \cite{Cazenave2003}. 
Define %the scaling
\be\label{3.13}V_{\rho_j}^j=\rho_jV^j(x-x_n^j)\ \ {\rm with}\
\  \rho_j=\frac{1}{\|V^j\|_2}\ne 0. \ \ee
Then  $\|V_{\rho_j}^j\|_2=1$ and
\begin{equation*}
 E_{\mu}(V_{\rho_j}^j)
=\rho_j^2 E_{\mu}(V^j(x-x_n^j))-\frac{\rho_j^2(\rho_j^{p-1}-1)}{p+1}\|V^j(x-x_n^j)\|_{p+1}^{p+1},
\end{equation*}
which implies 
\be\label{3.14}E_{\mu}(V^j(x-x_n^j))=\frac{E_{\mu}(V_{\rho_j}^j)}{\rho_j^2}+\frac{\rho_j^{p-1}-1}{p+1}\|V^j(x-x_n^j)\|_{p+1}^{p+1}.\ee

Similarly, $E_{\mu}(r_n^l)$ can be estimated as follows:
\be\label{000}
E_{\mu}(r_n^l)=\|r_n^l\|_2^2E_{\mu}(\frac{r_n^l}{\|r_n^l\|_2})+\frac{\left(\frac{1}
{\|r_n^l\|_2}\right)^{p-1}-1}{p+1}\|r_n^l\|_{p+1}^{p+1}\geq
\|r_n^l\|_2^2E_{\mu}(\frac{r_n^l}{\|r_n^l\|_2})
\ee
as $n, l\rightarrow+\infty$. On the one hand, it follows from  the definition of $m_{\mu}$
 that \be\label{3.15}E_{\mu}(V_{\rho_j}^j)\geq m_{\mu}\ \ \ {\rm and }\ \ \
E_{\mu}(\frac{r_n^l}{\|r_n^l\|_2})\geq m_{\mu}\,.
\ee 
Meanwhile, since
$\sum\limits_{j=1}^{l}\|V^j(x-x_n^j)\|_2^2$ is convergent,  there
exists $j_0\geq 1$ such that
\be\label{3.16}\inf\limits_{j\ge 1}\frac{\rho_j^{p-1}-1}{p+1}=\frac{1}{p+1}\left(\frac{1}{\|V^{j_0}\|_2^{p-1}}-1\right).
\ee
Substituting (\ref{3.14})-(\ref{3.16}) into (\ref{3.12}), we obtain the following estimate
as $n\rightarrow +\infty$ and $l\rightarrow +\infty$
   \begin{align}
   \label{3.17}
   E_{\mu}(v_n)=&\sum\limits_{j=1}^{l}\left(\frac{E_{\mu}(V_{\rho_j}^j)}{\rho_j^2}
   +\frac{\rho_j^{p-1}-1}{p+1}\|V^j(x-x_n^j)\|_{p+1}^{p+1}\right)+E_{\mu}(r_n^l)+o_{n,l}(1)\notag\\
   \ge&  \sum\limits_{j=1}^{l}\frac{m_{\mu}}{\rho_j^2}
+\inf\limits_{j\ge 1}\frac{\rho_j^{p-1}-1}{p+1}(\sum\limits_{j=1}^{l}\|V^j(x-x_n^j)\|_{p+1}^{p+1})+\|r_n^l\|_2^2m_{\mu} +o_{n,l}(1)\notag\\
\ge&  \sum\limits_{j=1}^{l}\frac{m_{\mu}}{\rho_j^2} +\|r_n^l\|_2^2\, m_{\mu}+ \frac{C_0}{p+1}\left(\frac{1}{\|V^{j_0}\|_2^{p-1}}-1\right)+o_{n,l}(1)\notag\\
  =& m_{\mu}+\frac{C_0}{p+1}\left(\frac{1}{\|V^{j_0}\|_2^{p-1}}-1\right)+o_{n,l}(1)
  \end{align}
for some constant $C_0>0$ independent of $n,\ell$, where we have also noted (\ref{2.7})-(\ref{2.37}). 

Now, taking $n\rightarrow +\infty$  and $l\rightarrow+\infty$ in (\ref{3.17}),
we obtain by (\ref{3.4})  
$$ \frac{C_0}{p+1}\left(\frac{1}{\|V^{j_0}\|_2^{p-1}}-1\right)\leq
0.$$ Then we must have $\|V^{j_0}\|_2^2\geq 1$. Equation (\ref{2.7}) implies there exists only one term
$V^{j_0}\neq 0$ in the decomposition (\ref{3.7}) such that
  $\|V^{j_0}\|_2^2=1$.
This means 
 \[
 \|v_n\|_2 \to \|V^{j_0}\|_2 \quad\text{ and}\quad  \|r_n^{j_0}\|_2\to 0 \qquad \text{as $n\to +\infty$.}
 \]
In order to show  $\|\Delta v_n\|_2 \to \|\Delta V^{j_0}\|_2$ as $n\rightarrow +\infty$, we note by \eqref{3.12} and \eqref{2.6} 
\begin{eqnarray}\label{3.255}
E_{\mu}(v_n)&=& \frac{1}{2} \int|\Delta  V^{j_0}|^{2}dx+\frac{\mu}{2} \int|\nabla V^{j_0}|^{2}dx
 -\frac {1}{p+1} \int |V^{j_0}|^{p+1}dx \nonumber  \\
 &+& \frac{1}{2} \int|\Delta  r_n^{j_0}|^{2}dx+\frac{\mu}{2} \int|\nabla r_n^{j_0}|^{2}dx
 -\frac {1}{p+1} \int |r_n^{j_0}|^{p+1}dx +o_n(1) \nonumber  \\
&=& E_{\mu}(V^{j_0}) +\underbrace{\frac{1}{2} \int|\Delta  r_n^{j_0}|^{2}dx+\frac{\mu}{2} \int|\nabla r_n^{j_0}|^{2}dx}_{nonnegative}+o_n(1).
\end{eqnarray}
This implies, in view of the definition of $m_\mu$, positivity of $\mu\ge 0$, (\ref{3.255}) and (\ref{2.8}) 
\begin{align*}
&E_{\mu}(V^{j_0})=m_{\mu}=\lim\limits_{n}E_{\mu}(v_n)\quad \\
&\lim_n  \|\Delta r_n^{j_0}\|_2=\lim_n  \|\nabla r_n^{j_0}\|_2=0\\
&\lim_n \Vert\Delta v_n\Vert_2=\| \Delta V^{j_0}\|_2\,.
\end{align*}
Therefore, by (\ref{2.5}), the weak convergence of $\{v_n\}$ in $H^2$ and the norm convergence $\norm{v_n}_{H^2}\to \norm{V^{j_0}}_{H^2}$
we conclude that $v_n(x+x_n^{j_0})\to V^{j_0}$ in $H^2$ 
and
$E_{\mu}(V^{j_0})=m_{\mu}$, namely, $m_{\mu}$ is achieved at  $V^{j_0}\in B_1$. Thus Case (i) is proved.
\vspace{0.3cm}

{\bf Case (ii):}   $\mu<0$. By Lemma \ref{lm3.2} we know that $m_{\mu}<0$ for all $\mu<0$. Let $\{v_n\}_{n=1}^{+\infty}\subset B_1$ be an arbitrary minimizing sequence of $m_{\mu}$,  then \eqref{00003} implies $\{v_n\}_{n=1}^{+\infty}$ is bounded in $H^2$ and satisfies (\ref{3.6}) as $n$ is sufficiently large. 
 We may assume $\{v_n\}$ is weakly convergent in $H^2$ without loss of generality. 
Now we claim that 
there exists a constant $\lambda_0:=\lambda_0(p,d,\|Q_p\|_2)>0$ such that given $-\lambda_0< \mu<0$, 
$\{v_n\}_{n=1}^{+\infty}$ is non-vanishing in the sense that for $n$ large enough
\begin{eqnarray}\label{3.00nv}
\int |v_n|^{p+1}dx \geq C_0>0
\end{eqnarray}
for some constant $C_0$.

Indeed, from \eqref{GN2} we have 
\begin{eqnarray}\label{3.01}
E_{\mu}(u)\geq \dfrac{1}{2}\|\Delta u\|_2^2 + \dfrac{\mu}{2}\|\Delta u\|_2 -C_{p,d}\|\Delta u\|_2^{\frac{(p-1)d}{4}}, \quad \forall u\in B_1.
\end{eqnarray}
where $C_{p,d}:=%\frac{B_{p,d}}{p+1}
\frac{1}{2\norm{Q_p}_2^{p-1}}$.
Define the function for all $k> 0$
\begin{align}\label{f_k(y)}
f_{k}(y):=\frac{1}{2}y^2 - \frac{k}{2}y -C_{p,d}\,y^{\frac{(p-1)d}{4}}, \quad \forall y\geq 0,
\end{align} 
noting that $f_k(0)=0$ and $0<\frac{(p-1)d}{4}<2$ if $p<1+\frac{8}{d}$. 
Let $m_0$ be given in \eqref{VP0}, then $m_0<0$ by Lemma \ref{lm3.4}. Observe that there exist unique constants
$y_{1,k}:=y_{1,k}(p,d,\|Q_p\|_2)>0$ and $y_{2,k}:=y_{2,k}(p,d,\|Q_p\|_2)>0$, with $y_{1,k}<y_{2,k}$ such that
\begin{eqnarray}\label{eqn:m0-fmu}
{m_0}> f_{k}(y)  \Longleftrightarrow  y\in (y_{1,k}, y_{2,k}).
\end{eqnarray}
Define \begin{equation}\label{lam0-maxmin}
\lambda_0:=\sup_{k>0}( \min\{k, y_{1,k}\}), 
\end{equation}%for all $k>0 
where we note that $k\mapsto \min\{k, y_{1,k}\}$ is continuous. 
%\begin{align}\label{e:fmu}
%&f_\mu(y):=\frac{1}{2}y^2 +\frac{\mu}{2}y -B_{p,d}\, y^{\frac{(p-1)d}{4}}, \quad \forall y\geq 0.\end{align}  
%\footnote{we could choose $y_1,y_2$ s.t.  $f_\mu(y)\le \frac{1}{2}y^2- B_{p,d} \,y^{\frac{(p-1)d}{4}}\le {m_0}<0$ on an interval $[y'_1,y'_2]$ too, however it won't work for the ineq. below $m_0\ge m_\mu ....\ge f(\norm{\De v_n}_2)+o(1)$ }
%\edz{note the convexity of f, or $f''$ is not evident on the interval $[-1,2]$ or $[0,2]$ say}
 Evidently $\lambda_0>0$ depends on $p,d$ and $\|Q_p\|_2$ only. Given  $-\lambda_0<\mu<0$, by the definition of $\lam_0$
 there exists $k_0$ such that $-\lambda_0<-\min(k_0,y_{1,k_0}) <\mu<0$. Since $\{v_n\}\subset B_1$ is a minimizing sequence of $m_\mu$, 
we have by \eqref{3.01}
\begin{eqnarray*}
m_0> m_{\mu} &=& E_{\mu}(v_n)+o_n(1)\\
&\ge& \dfrac{1}{2}\|\triangle v_n\|_2^2 + \dfrac{\mu}{2}\|\nabla  v_n\|_2 - C_{p,d}\|\De v_n\|_2^{\frac{(p-1)d}{4}}+o_n(1) \\
&\ge& \dfrac{1}{2}\|\triangle v_n\|_2^2 - \dfrac{k_0}{2}\|\De v_n\|_2 - C_{p,d}\|\De v_n\|_2^{\frac{(p-1)d}{4}}+o_n(1)\\
&=& f_{k_0}(\|\De v_n\|_2)+o_n(1),
\end{eqnarray*}
where $o_n(1)\to 0$ as $n\to +\iy$.  
%\footnote{In  the above estimation, we have used $\mu> -k_0$ and $\|\nabla  v_n\|_2\le \|\De  v_n\|_2$ for $v_n\in B_1$.} 
Hence for sufficiently large $n$, there holds $m_0>f_{k_0}(\Vert \Delta v_n\Vert_2)$.
In virtue of \eqref{eqn:m0-fmu}, we obtain that  $\|\De v_n\|_2> y_{1,k_0}>0$ for large $n$. %taking $n$ large enough if necessary 
 Thus, if $-\lambda_0<\mu<0$, we have for $n$ large enough 
\begin{eqnarray}
\frac{1}{p+1}\int |v_n|^{p+1}dx&=&\frac12\|\triangle v_n\|_2^2+\dfrac{\mu}{2}\|\nabla v_n\|_2^2-E_{\mu}(v_n)\notag\\
 &\ge&  \frac{1}{2} \|\triangle v_n\|_2(\|\triangle v_n\|_2+\mu)-\dfrac{m_{\mu}}{2}\notag\\
 &\ge& \frac{\lambda_0+\mu }{2} \|\De v_n\|_2-\dfrac{m_{\mu}}{2}\notag\\
&>& -\frac{m_{\mu}}{2}>0,\label{Mmu+}
\end{eqnarray}
which establishes \eqref{3.00nv}. %\footnote{where we have used $-\mu<y_{1,k_0}$} 
Therefore, by H\"older inequality  we conclude that $\{v_n\}\subset B_1$ is non-vanishing in $L^q(\R^d)$ for all $q\in(2,\frac{2d}{(d-4)^{+}})$. 

Finally, we complete the proof of Case (ii) by  showing the pre-compactness of $\{v_n\}$ using the profile decomposition (\ref{2.5}) to (\ref{2.37}). 
 This proceeds the same as  in Case (i),  %role of $\mu\in \mathbb{R}$ is not essential,
  except for the part (\ref{3.255}) in proving
   $\|\Delta v_n\|_2 \to \|\Delta V^{j_0}\|_2$ as $n\rightarrow +\infty$. Indeed, since $\mu<0$, we have by \eqref{3.12} and \eqref{2.6}
\begin{align}\label{E-mu-r_n}
E_{\mu}(v_n) %=\frac{1}{2} \int|\Delta  V^{j_0}|^{2}+\frac{\mu}{2} \int|\nabla V^{j_0}|^{2} -\frac {1}{p+1} \int |V^{j_0}|^{p+1}  \\
 %+& \frac12\int|\Delta  r_n^{j_0}|^{2} +\frac{\mu}{2} \int|\nabla r_n^{j_0}|^2 -\frac {1}{p+1} \int |r_n^{j_0}|^{p+1} +o_n(1)  \\
=&E_{\mu}(V^{j_0}) + \frac{1}{2} \int|\Delta  r_n^{j_0}|^{2}dx+\frac{\mu}{2} \int|\nabla r_n^{j_0}|^{2}dx +o_n(1)\notag\\
\ge&E_{\mu}(V^{j_0}) + \frac{1}{2} \int|\Delta  r_n^{j_0}|^{2}dx+\frac{\mu}{2} \Vert \De r_n^{j_0}\Vert_2 \norm{r_n^{j_0}}_2 +o_n(1)\notag\\ 
=&E_{\mu}(V^{j_0}) + \frac{1}{2} \int|\Delta  r_n^{j_0}|^{2}dx +o_n(1), 
\end{align}
where we note $\|\Delta r_n^{j_0}\|_2\|r_n^{j_0}\|_2\to 0$ since $\|\De r_n^{j_0}\|_2$ is bounded and $\norm{r_n^{j_0}}_2\to 0$. 
%\edz{this part also applies to the case $\mu>0$ and so it can be omitted} 
It follows that 
 \begin{align*}
&E_{\mu}(V^{j_0})=m_{\mu}=\lim\limits_{n\to +\infty}E_{\mu}(v_n)\quad\\ 
&\lim_n\|\Delta r_n^{j_0}\|_2= 0\quad\text{and}\quad
\lim_n\|\Delta v_n\|_2= \|\Delta V^{j_0}\|_2\,.
\end{align*} 
And so, since $\{v_n\}$ is weakly convergent in $H^2$, we must have $\lim_n v_{n}(\cdot+x^{j_0}_n)=V^{j_0}$ in $H^2$ strongly. 
% Since in this procedure the role of $\mu\in \mathbb{R}$ is not essential, the  proof goes the same as in Case (i).  For simplicity we omit it 
  Therefore, we have proved Proposition \ref{M}. 
\end{proof}
\begin{remark} The non-vanishing property \eqref{3.00nv} is one key estimate in the proof
for the case $\mu<0$. % technically more challenging to deal with (\ref{3.00nv}). 
Observe that the argument for proving the non-vanishing property (\ref{001}) in the case $\mu\ge 0$ 
does not work for the case $\mu<0$. We overcome the difficulty by introducing the $f_\mu$ function (\ref{f_k(y)}), which enables us to
obtain a lower bound for $\norm{\De v_n}_2$ and then a lower bound for $\norm{v_n}_{p+1}^{p+1}$.
In the proof for $\mu<0$, we can also pick $\lambda_{0,1}:=\min\{1, y_{1,1}\}$ in place of the notion $\lam_0$ defined in (\ref{lam0-maxmin}). 
However, $-\lam_0=-\lam_0(p,d,\norm{Q_p}_2)$ gives an optimal lower bound for $\mu<0$.
%\begin{eqnarray*}  m_0> m_{\mu} &=& E_{\mu}(v_n)+o_n(1)\\
%&\ge \frac{1}{2}\|\Delta v_n\|_2^2 + \frac{\mu}{2}\|\nabla  v_n\|_2^2 - C(p,\|Q_p\|_2)\|\Delta v_n\|_2^{\frac{(p-1)d}{4}}+o_n(1) \\
%&\ge \frac{1}{2}\|\Delta v_n\|_2^2 - \frac{1}{2}\|\Delta v_n\|_2 - C(p,\|Q_p\|_2)\|\Delta v_n\|_2^{\frac{(p-1)d}{4}}+o_n(1)\\
%&=& f_1(\|\Delta v_n\|_2)+o_n(1), \end{eqnarray*}  here the last step estimate  use the fact $\mu\ge -\lambda_0\ge -1 
%This together with \eqref{3.02}, implies  \|\Delta v_n\|_2> y_1\ge\lam_0>0, taking n large  if necessary. Hence, assuming  -\lam_0<\mu<0  there holds for n large 
%\begin{align*} %\frac{1}{p+1}\int |v_n|^{p+1}dx=&\frac12\|\Delta v_n\|_2^2+\frac{\mu}{2}\|\nabla v_n\|_2^2-E_{\mu}(v_n)\\
 %\ge&  \dfrac{1}{2} \|\Delta v_n\|_2(\|\Delta v_n\|_2+\mu)-\dfrac{m_{\mu}}{2}\\
 %\ge& \frac{\lambda_0+\mu }{2} \|\Delta v_n\|_2-\dfrac{m_{\mu}}{2}\\
 %\ge& -\frac{m_{\mu}}{2}>0,  \end{align*} which verifies \eqref{3.00nv}
\end{remark}

 For $\mu<0$,  a careful examination of the proof in the Case (ii) of the above proposition  shows that there is an alternative sufficient condition  %better and sharpest than the condition
that ensures the existence of ground states for (\ref{VP}).
\begin{proposition}\label{pr:mu2} Let $1<p<1+\frac8d$ and $\mu<0$.  Suppose $m_\mu<-\frac{\mu^2}{8}$. 
Then any minimizing sequence of $m_{\mu}$ is pre-compact in $H^2$. Moreover, (\ref{VP}) admits a minimizer, namely,  ground state solution. 
\end{proposition}
\begin{proof}[Outline of the proof]
 First, notice that $m_\mu<m_0<0$ if $\mu<0$, thus the condition $m_\mu<-\frac{\mu^2}{8}$ is satisfied for $\mu\in (-\Lambda_0, 0)$ with some small $\Lambda_0>0$.   
  %\footnote{There is sharp refinement if assuming $m_\mu<-\frac{\mu^2}{8}$ in that case we attain the end point of this issue} (assuming $(v_n)\subset B_1$ never touch $m_\mu) 
The proof of Proposition \ref{pr:mu2} is similar to that of Proposition \ref{M}, but without using the $f_k$ functions. 
The only point of check is the proof of the non-vanishing condition (\ref{3.00nv}), which we show as follows. 
Let $\eta>0$ be arbitrary. Then for $n$ large enough, $m_\mu\le E_{\mu}(v_n)<m_\mu+\eta$, and so,
\begin{align}
&\frac{1}{p+1}\int |v_n|^{p+1}=\frac12\|\De v_n\|_2^2+\dfrac{\mu}{2}\|\nabla v_n\|_2^2-E_{\mu}(v_n)\notag\\
 >&\frac{1}{2} \Vert\De v_n\Vert_2^2 +\frac{\mu}{2}\Vert \De v_n\Vert_2-(m_{\mu}+\eta)\notag\\
\ge& -\frac{\mu^2 }{8}-m_{\mu}-\eta >0\label{vn-p+eta}
\end{align} 
if choosing $\eta$ small such that $0<\eta< -m_{\mu}-\frac{\mu^2 }{8} $. Here we have noted that 
the minimum of the function $g(y):=\frac{1}{2} y^2 +\frac{\mu}{2}y$ is given by $-\frac{\mu^2 }{8}$.
Therefore the non-vanishing condition (\ref{vn-p+eta}) is established. This concludes the outline of the proof of Proposition \ref{pr:mu2}. %\hfill$\Box$
\end{proof}
%imply $(v_n)$ non-vanishing  $\norm{V^j}_2\not\to 0 
%\begin{align*} &\mu+\Vert \De v_n\Vert_2>0.  \end{align*}  however this might require numerical verification  not show link to  Q_p

Now we apply Proposition \ref{M} to complete  the proof of Theorem \ref{th1} on the orbital stability by following 
 a standard concentration compactness argument in \cite{Cazenave2003,CazenaveLions1982}. %\end{document}

\begin{proof}[{\bf Proof of Theorem \ref{th1}}] We shall prove the orbital stability by contradiction. 
Write $\psi(t)=\psi(t,\cdot)$.
First, we claim that if $\mu\in\R$ and $1<p<1+\frac8d$,  then $\{\|\psi(t)\|_{H^2}\}$ is bounded for all $t\in I$. 
According to Proposition \ref{local},  the solution $\psi$ of (\ref{BNLS})-(\ref{BNLS1}) exists globally in time. 
To prove the claim, we divide our discussions in two cases. 
%\begin{enumerate} \item $\mu\ge 0$   \item $\mu<0$.  \end{enumerate} 
Case (i): $\mu\ge 0$. From \eqref{GN2} and Proposition \ref{local}  we have 
 for all $t\in I$ (the maximal time interval)
\[\begin{array}{lll} \vspace{0.3cm}
 E_{\mu}(\psi_0)=E_{\mu}(\psi)&\geq \frac12 \|\Delta \psi(t)\|_2^2+\frac{\mu}{2}\|\nabla
\psi(t)\|_2^2- C \|\psi(t)\|_2^{\frac{(4-d)p+4+d}{4}}\|\Delta
\psi(t)\|_2^{\frac{(p-1)d}{4}}\\
 &\ge
(\frac12-\veps) \|\Delta \psi(t)\|_2^2-C(\varepsilon, p,d,\|\psi_0\|_2)\end{array}
\]
for all $0<\varepsilon<\frac12$. Thus,  $\{\|\psi(t)\|_{H^2}\}$ is bounded for all $t\in I$. 

Case (ii)  $\mu<0$. Similarly, %for all $t\in I$ (the maximal 
 we deduce that  for any $0<\varepsilon<\frac{2}{4-\mu}$
\[\begin{array}{lll} \vspace{0.2cm}
 E_{\mu}(\psi_0)=E_{\mu}(\psi)&\geq \frac12 \|\Delta \psi(t)\|_2^2+\frac{\mu}{2}\|\nabla
\psi(t)\|_2^2- C \|\psi(t)\|_2^{\frac{(4-d)p+4+d}{4}}\|\Delta
\psi(t)\|_2^{\frac{(p-1)d}{4}}\\
 &\geq (\frac12+\frac{\mu \varepsilon}{4}-\varepsilon)\|\Delta \psi(t)\|_2^2+\frac{\mu}{4\veps}\|\psi_0\|_2^2-C(\varepsilon, p,d,
\|\psi_0\|_2),\end{array}\] where $\frac12+\frac{\mu \varepsilon}{4}-\varepsilon>0$. 
Then %by $\|\psi(t)\|_2=\|\psi_0\|_2$ and the interpolation, 
we see  $\{\|\psi(t)\|_{H^2}\}$ is bounded  for all $t\in I$.
This verifies  the claim and hence $\psi$ exists globally in time by the blowup alternative assertion in Proposition \ref{local}. 

Secondly, we  prove the orbital stability for $\mathcal{M}_\mu$. %for Eq. (\ref{BNLS})
  Assume by contradiction that $\mathcal{M}_{\mu}$ is not orbitally stable, 
%then there exist a $\varepsilon_0>0$, and solutions $\psi_n(t_n)$ with the initial datum $\psi_n(0)$, s.t.
%\begin{eqnarray}\label{defi-stable}
%\inf_{\phi\in \mathcal{M}_{\mu}}\|\psi_n(0) - \phi\|_{H^2}\to 0,\ \mbox{ and } \inf_{\phi\in \mathcal{M}_{\mu}}\|\psi_n(t_n) - \phi\|_{H^2}\geq \varepsilon_0.
%\end{eqnarray} 
then there exist  $\varepsilon_0>0$ and a sequence of initial data $\{\psi_0^n\}_{n=1}^{+\infty}$ such that
\be\label{3.22}\inf\limits_{u\in
\mathcal{M}_{\mu}}\|\psi_0^n-u\|_{H^2}<\frac1n\ee 
and there exists a sequence
$\{t_n\}_{n=1}^{+\infty}$ such that the corresponding solution
sequence $\{\psi_n(t_n)\}_{n=1}^{+\infty}$ satisfies
 \be\label{3.23}\inf\limits_{u\in \mathcal{M}_{\mu}}\|\psi_n(t_n)-u\|_{H^2}\geq\varepsilon_0.\ee
Note from the conservation laws that as $n\rightarrow +\infty$
\begin{eqnarray*}
\left\{\begin{matrix}
\int |\psi_n(t_n)|^2dx=\int |\psi_0^n|^2dx\rightarrow \int |u|^2dx=1,\\
E_{\mu}(\psi_n(t_n))=E_{\mu}(\psi_0^n)\rightarrow E_{\mu}(u)=m_{\mu}.
\end{matrix}\right.
\end{eqnarray*}
Let $\varphi_n(t_n):= \rho_n \cdot \psi_n(t_n)$ with $\rho_n= 1/ \|\psi_n(t_n)\|_2$, then $\varphi_n(t_n)\in B_1$ and $\rho_n \to 1$. In particular,  $\{\varphi_n(t_n)\}_{n=1}^{+\infty}$ is a minimizing sequence of $m_{\mu}$. From Lemma \ref{lm3.2} and
Proposition \ref{M},  we see that under the assumptions of Theorem \ref{th1}, 
there exists a minimizer
$w\in B_1$ such that $ \|\varphi_n(t_n)-w\|_{H^2}\rightarrow 0$ as $n\to +\infty$. 
That means
 \begin{equation*}
  \|\psi_n(t_n)-w\|_{H^2}\rightarrow 0\  \ {\rm as}\ \ n\rightarrow +\infty\, ,
  \end{equation*}
which  contradicts  (\ref{3.23}). This  completes the proof.
\end{proof}

\begin{proof}[{\bf Proof of Corollary \ref{th1.2}}] Let $\{\mu_k\}_{k=1}^{+\infty}$ be a sequence with $\mu_k\to 0$ as $k\to+\infty$, and $\{u_k\}_{k=1}^{+\infty}\subset B_1$ be a sequence of minimizers for $m_{\mu_k}<0$, namely, 
\[
E_{\mu_k}(u_k)=m_{\mu_k}<0, \quad  \forall k\in \mathbb{N}.
\]
By the continuity in Lemma \ref{lm3.1} $(b)$ and Lemma \ref{lm3.4}, we see $m_{\mu_k} \to m_0<0$ as $k\to +\infty$. Then $\{u_k\}_{k=1}^{+\infty}\subset B_1$ is a minimizing sequence of $m_0$.  
We claim that there exists a subsequence $\{u_{k_j}\}\subset \{u_k\}$ such that
\begin{itemize}
  \item [(a)] $\{u_k\}_{k=1}^{+\infty}\subset B_1$ is bounded in $H^2$\, ;
  \item [(b)] $u_{k_j}\underset{j}\to u_0$ in $H^2$ for some $u_0\in B_1$.
\end{itemize}
Indeed, we deduce from \eqref{00002} and \eqref{00003}  that
\begin{eqnarray}\label{3.35}
m_{\mu_k}=E_{\mu_k}(u_k)\geq \frac12 \|\Delta u_k\|_2^2
-C \|\Delta u_k\|_2^{\frac{(p-1)d}{4}}, \ \mbox{ if } \mu_k>0,
\end{eqnarray}
and \begin{eqnarray}\label{3.351}
m_{\mu_k}=E_{\mu_k}(u_k)\geq \frac{1}{2}\|\Delta u_k\|_2^2 + \frac{\mu_k}{2} \|\Delta u_k\|_2  
-C \|\Delta u_k\|_2^{\frac{(p-1)d}{4}},\ \mbox{ if } \mu_k<0
\end{eqnarray}
for some constant $C>0$ independent of $k\in \mathbb{N}$.  Note that $0<\frac{(p-1)d}{4}<2$ if $1<p<1+\frac{d}{8}$. It follows from \eqref{3.35} and \eqref{3.351} %and $m_{\mu_k}\to m_0<0$ 
that $\{\|\Delta u_k\|_2\}$ is bounded. %Further by the inequality $\|\nabla u\|_2^2\leq \|\Delta u\|_2\|u\|_2$, $\{\|\nabla u_k\|_2\}_{k=1}^{+\infty}$ is bounded. 
Thus $(a)$ is verified. 

Now that $\{u_k\}\subset B_1$ is bounded in $H^2$, there exist a weakly convergent subsequence 
$\{u_{k_j} \}_j$ and some $u_0$ in $H^2$ such that  $u_{k_j} \rightharpoonup u_0$ in $H^2$ as $j\to +\iy$.
 If $\{u_{k_j}\}$ is non-vanishing in $L^{p+1}$, namely, there exists a constant $C_0$  such that
\begin{eqnarray}\label{3.36}
\int |u_{k_j}|^{p+1}\geq C_0>0
\end{eqnarray}
for $j$ sufficiently large, then we can apply profile decomposition argument as in the proof of Proposition \ref{M} 
 to show that  $u_{k_j}\to u_0$ in $H^2$. Thus it  remains  to verify \eqref{3.36}. 
 But this can be proved the same way as (\ref{001}) and (\ref{3.00nv})   %any $\mu>0$ and \mu<0\to -0
as soon as  $\mu_{k_j}$ is sufficiently close to $0$.
% we note that for $k\in\mathbb{N}$ large, \frac{1}{p+1}\int |u_k|^{p+1}dx=\frac{1}{2}\|\Delta u_k\|_2^2+\frac{\mu_k}{2}\|\nabla u_k\|_2^2-E_{\mu_k}(u_k)\geq - \frac{1}{2}m_{\mu_k} \to -\frac{1}{2}m_0>0,$$
%which proves \eqref{3.36}. At this point, the proof is complete.
\end{proof}

\begin{proof}[{\bf Proof of Theorem \ref{th1.1}}] 
We assume by contradiction that for $\mu >\mu_0$, there exists a minimizer  $v_0\in B_1$ such that $E_{\mu}(v_0)=m_{\mu}$. By Lemma \ref{lm3.2} (2), $E_{\mu}(v_0)=m_{\mu}=0$.  Thus from the definition of $E_{\mu}(u)$, we have
$$\dfrac{\mu-\mu_0}{2}\|\nabla v_0\|_2^2=E_{\mu}(v_0)-E_{\mu_0}(v_0)\leq 0-m_{\mu_0}=0,$$
which is a contradiction %if $\mu>\mu_0
since $\|\nabla v_0\|_2\neq 0$. Therefore, we have shown $\mathcal{M}_{\mu}=\emptyset$ for all $\mu \in (\mu_0, +\infty)$.
\end{proof}
%\end{document}

\renewcommand{\theequation}{\thesection.\arabic{equation}}
 \setcounter{equation}{0}
\section{Construction of standing waves in the $L^2$ critical case}\label{s:L2-critic}

In this section, we address the $L^2$-critical case $p=1+\frac{8}{d}$ by
consider the following minimization problem
proposed in (\ref{1.7}):  Given $\mu\in \mathbb{R}$ and $b>0$,
\begin{eqnarray*}
m_{\mu, b} := \inf_{u\in B_1}E_{\mu,b}(u),
\end{eqnarray*}
where %$E_{\mu,b}$ 
\begin{eqnarray*}
E_{\mu,b}(u)=\frac{1}{2}\|\Delta u\|_2^2 + \frac{\mu}{2}\|\nabla u\|_2^2 - \frac{b}{2+\frac{8}{d}}\int |u|^{2+\frac{8}{d}}dx
\end{eqnarray*}
as defined in (\ref{1.8}).  Recall that when $p=1+\frac{8}{d}$, the Gagliardo-Nirenberg inequality \eqref{GN2} reads
\begin{eqnarray}\label{4.0}
\int |u|^{2+\frac{8}{d}}dx \leq \dfrac{1+\frac{4}{d}}{\|Q^*\|_2^{\frac{8}{d}}}\|u\|_2^{\frac{8}{d}}\|\Delta u\|_2^2 \qquad \forall u\in H^2,
\end{eqnarray}
where $Q^*$ is a ground state of \eqref{1.4} and the equality holds  if and only if 
$u$ is a minimizer of (\ref{JpdU}) or equivalently, $u$ solves (\ref{1.4}). 
Recall $b^*=\|Q^*\|_2^{\frac{8}{d}}$ from (\ref{4.00}), then similar to the proof of \cite[Theorem 1.2]{BCGJ}, %as with $\mu>0$, 
 we can easily obtain the following lemma by applying  Pohozaev identity (\ref{e:Qu}). 
\begin{lemma}\label{lm4.0} Let $\mu\ge  0$ and $p=1+\frac8d$. Then 
\begin{equation}\label{4.3}
\begin{cases}
m_{\mu, b}=0,         & 0<b\leq b^*,\\
m_{\mu, b}=-\infty,   &  b>b^*.
\end{cases}
\end{equation}
Moreover, if $\mu> 0$, then for all $b\in (0,b^*]$, the functional $E_{\mu,b}(u)$ has no  critical point on $B_1$, that is, $m_{\mu,b}$ can not be achieved for any $b>0$. 
If $\mu=0$, then a ground state solution of (\ref{1.7}) exists if and only if  $b=b^*$.
\end{lemma}

Hence in the remaining of the section we mainly consider the case $\mu<0$. 
\begin{lemma}\label{lm4.1}
Let $\mu<0$ and $p=1+\frac8d$.  Then
\begin{equation}\label{4.4}
\begin{cases}
-\infty<m_{\mu, b}<0, & 0<b< b^*,\\
m_{\mu, b}=-\infty,     &  b\geq b^*.
\end{cases}
\end{equation}
\end{lemma}

\begin{proof}
First, by \eqref{4.0} and the inequality $\|\nabla u\|_2^2\leq \|\Delta u\|_2 \|u\|_2$, we have for all $u\in H^2$ 
\begin{eqnarray*}
E_{\mu,b}(u)\ge%&\frac{1}{2}\Big[1-\dfrac{b}{b^*} \Big] \|\Delta u\|_2^2 + \frac{\mu}{2}\|\nabla u\|_2^2 
 \frac{1}{2}\Big[1-\dfrac{b}{b^*}\norm{u}_2^{\frac8d} \Big]\|\Delta u\|_2^2 + \frac{\mu}{2}\|\Delta u\|_2\|u\|_2\,.
\end{eqnarray*}
Thus
\begin{eqnarray}\label{4.5}
E_{\mu,b}(u)\geq \frac{1}{2}\Big[1-\dfrac{b}{b^*} \Big]\|\Delta u\|_2^2 + \frac{\mu}{2}\|\Delta u\|_2,\quad \forall u\in B_1,
\end{eqnarray}
which implies that $m_{\mu,b}\neq -\infty$ for every $0<b<b^*$. 
 To show that $m_{\mu,b}<0$ if $0<b<b^*$, we consider the scaling
$v^{\rho}=\rho^{\frac d2} v_0(\rho x)$, where $v_0\in B_1$ is given and $\rho>0$ is an arbitrary constant. Then $v^{\rho}\in B_1$ for any $\rho>0$  and
\begin{eqnarray}\label{4.6}
E_{\mu,b}( v^{\rho} )&=&\frac{\rho^4 }{2} \|\Delta v_0\|_2^2+\frac{\mu \rho^2}{2}\|\nabla
v_0\|_2^2-\frac{\rho^{4}b}{2+\frac{8}{d}}\|v_0\|_{2+\frac{8}{d}}^{2+\frac{8}{d}}\\
&<& \frac{\rho^4}{2} \|\Delta v_0\|_2^2+\frac{\mu \rho^2}{2}\|\nabla \nonumber
v_0\|_2^2\,.
\end{eqnarray}
Then by taking $\rho=\rho_0:=\frac{ \sqrt{-\mu} }{2}\cdot \frac{\|\nabla v_0\|_2} {\|\Delta v_0\|_2}$,  
we see % positive constant $C_0$ such that
$E_{\mu,b}( v^{\rho} )<\frac{\mu \rho^2}{4}\|\nabla v_0\|_2^2:= -C_0<0$. This proves that $-\infty< m_{\mu,b}<0$ if $0<b<b^*$.

If $b\geq b^*$, replacing $v_0$ in \eqref{4.6} by $Q^*$ with $Q^*$ given as in \eqref{4.0}, we have
\begin{eqnarray*}
E_{\mu,b}( (Q^*)^{\rho} )&=& \frac{\rho^4}{2} \|\Delta Q^*\|_2^2+\frac{\mu \rho^2}{2}\|\nabla
Q^*\|_2^2-\frac{\rho^{4}b}{2+\frac{8}{d}}\|Q^*\|_{2+\frac{8}{d}}^{2+\frac{8}{d}} \\
&=& \frac{\rho^4}{2}\Big[1-\dfrac{b}{b^*} \Big] \|\Delta Q^*\|_2^2 + \frac{\mu \rho^2}{2}\|\nabla Q^*\|_2^2\\
&\leq& \frac{\mu \rho^2}{2}\|\nabla Q^*\|_2^2,\quad \mbox{since } b\geq b^*.
\end{eqnarray*}
In view of $\mu<0$, the preceding inequality implies that $E_{\mu,b}( (Q^*)^{\rho} )\to -\infty$ as $\rho \to +\infty$. Thus $m_{\mu,b}=-\infty$ as $b\geq b^*$.
\end{proof}

Lemma \ref{lm4.1} informs that  $m_{\mu,b}<0$ provided $0<b<b^*$. This allows us to prove the following proposition on the existence of $\mathcal{M}_{\mu,b}$
by following the idea of the proof of Proposition \ref{M}, Case $(ii)$ and Proposition \ref{pr:mu2}.
  \begin{proposition}\label{prop4.1} Let $p=1+\frac8d$.
For any given $\mu\in (-\dfrac{4\|\nabla Q^*\|^2_2}{\|Q^*\|^2_2},0)$, define as in \eqref{4.00}
\begin{eqnarray*}%\label{4.000}
b^*=\|Q^*\|_2^{\frac{8}{d}},\quad  b_*=b^*\Big[1+ \dfrac{\|Q^*\|_2^2}{4\|\Delta Q^*\|_2^2}(\mu^2+\dfrac{4\|\nabla Q^* \|_2^2}{\|Q^*\|_2^2}\mu)\Big].
\end{eqnarray*}
Then  $0< b_*<b^*$, and for all $b\in (b_*, b^*)$, any minimizing sequence of $m_{\mu,b}$ is pre-compact in $H^2$. Moreover, there exists some  $u\in B_1$  such that
\begin{align}\label{mmb-E}
m_{\mu,b}=E_{\mu,b}(u),
\end{align}
namely $\mathcal{M}_{\mu,b} \neq \emptyset$.
\end{proposition}
One  main ingredient in the proof is to establish the non-vanishing property for any minimizing sequence of (\ref{1.7}). 
If $p<1+\frac8d$, $\mu<0$, one can prove  (\ref{3.00nv}) based on the fact that $m_0<0$, see 
 (\ref{3.6}) and (\ref{Mmu+}). %in  the proof of  Case $(ii)$ in Proposition \ref{M}
However, if $p=1+\frac8d$, we know $m_{0, b}=0$, for all $ 0<b< b^*$ in view of Lemma \ref{lm4.0}.
 This would present an obstacle for showing the non-vanishing property (\ref{4.10}) or (\ref{4.11}) as below
for any minimizing sequence of $m_{\mu,b}$.  
  To overcome this obstacle, we need to proceed differently. Following the spirit of Proposition \ref{pr:mu2}
  we shall prove the non-vanishing property under the natural condition (\ref{4.9}), 
as is required in the following two lemmas. 

\begin{lemma}\label{lm4.5} 
Let $\mu<0$ and  $0<b<b^*$ satisfy 
\begin{eqnarray}\label{4.9}
m_{\mu,b}<-\dfrac{\mu^2}{8}.
\end{eqnarray} 
Let  $\{u_n\}_{n=1}^{+\infty}\in B_1$ be an arbitrary minimizing sequence of $m_{\mu,b}$. Then there exists a constant 
$C_0$ independent of $n$ such that
\begin{eqnarray}\label{4.10}
\int |u_n|^{2+\frac{8}{d}}dx\geq C_0>0\qquad
\end{eqnarray}
 for $n$ large enough. \end{lemma}

\begin{proof} We only need to show $\liminf_n  \int |u_n|^{2+\frac{8}{d}}dx\ne 0$.
Assume the contrary. Then there is a subsequence of $\{u_n\}$ (still denoted by $\{u_n\}$) 
so that  $\lim_n\int |u_n|^{2+\frac{8}{d}}dx=0$. 
Then, from $m_{\mu,b}=\lim\limits_{n\to +\infty}E_{\mu,b}(u_n)$ we have
\[\|\Delta u_n\|_2^2 +\mu \|\nabla u_n\|_2^2\to 2m_{\mu,b} \qquad \text{as $n\to +\iy$}.
\]
By the inequality $\|\nabla u_n\|_2^2 \leq \|\Delta u_n\|_2\|u_n\|_2$ we deduce that
\[
\|\Delta u_n\|_2^2 +\mu \|\Delta u_n\|_2\leq \|\Delta u_n\|_2^2 +\mu \|\nabla u_n\|_2^2\to 2m_{\mu,b}\,.
\]
Since the function $y^2+\mu y$ has a minimum $-\dfrac{\mu^2}{4}$, by passing to the limit we obtain
\[
-\frac{\mu^2}{4}\le 2m_{\mu,b}\, ,
\]
which is a contradiction to (\ref{4.9}).  This proves the lemma.
\end{proof}

\begin{lemma}\label{lm4.6} Let $p=1+\frac8d$. Given $\mu\in (-\dfrac{4\|\nabla Q^*\|^2_2}{\|Q^*\|^2_2},0)$,  let $b_*$ and $b^*$ be defined in \eqref{4.00}. Then
\begin{itemize}
\item [(i)] $0< b_*<b^*$;
\item [(ii)] For all $b\in (b_*, b^*)$, $(\mu, b)$ satisfies \eqref{4.9}.  In particular, any minimizing sequence $\{u_n\}$ of $m_{\mu,b}$ is non-vanishing in the following sense:
There exists  some constant $C_0$ independent of $n$ such that
\begin{eqnarray}\label{4.11}
\int |u_n|^{q}dx\geq C_0>0,\qquad \forall q\in \left(2,\frac{2d}{(d-4)^{+}}\right)
\end{eqnarray}
when $n$ is large enough. 
\end{itemize}
\end{lemma}

\begin{proof} Indeed, by the definition of $b_*$, we easily observe that $0<b_*<b^*$ if and only if  
 $\mu\in (-\dfrac{4\|\nabla Q^*\|^2_2}{\|Q^*\|^2_2},0)$. %and \|\nabla Q^*\|_2^2\leq \|\Delta Q^*\|_2\|Q^*\|_2 
 Then $(i)$ is verified.

To show $(ii)$, %\begin{eqnarray*}
%E_{\mu,b}(Q^*) %=&\dfrac{1}{2}\|\Delta Q^*\|_2^2 +\dfrac{\mu}{2} \|\nabla Q^*\|_2^2-\dfrac{b}{2+\frac{8}{d}}\int|Q^*|^{2+\frac{8}{d}}dx \\
%=\dfrac{1}{2}(1-b)\|\Delta Q^*\|_2^2 +\dfrac{\mu}{2} \|\nabla Q^*\|_2^2\,. \end{eqnarray*}
define $v_0:=\frac{Q^*}{\|Q^*\|_2}$. Then $v_0\in B_1$ and 
we note from (\ref{Q-De-p}) that  
\[
E_{\mu,b}(v_0)=\dfrac{1}{2\|Q^*\|_2^2}\Big[(1-\dfrac{b}{b^*})\|\Delta Q^*\|_2^2 +\mu\|\nabla Q^*\|_2^2 \Big].
\]
We have for $0<b<b^*$,
\begin{eqnarray*}
 E_{\mu,b}(v_0) < -\dfrac{\mu^2}{8} %& \Longleftrightarrow & \dfrac{1}{2\|Q^*\|_2^2}\Big[(1-\dfrac{b}{b^*})\|\Delta Q^*\|_2^2 +\mu\|\nabla Q^*\|_2^2 \Big] < -\dfrac{\mu^2}{8} \\
&\Longleftrightarrow& \mu^2 +  \dfrac{4\|\nabla Q^*\|_2^2}{\|Q^*\|_2^2} \mu +  \dfrac{4}{\|Q^*\|_2^2}(1-\dfrac{b}{b^*})\|\Delta Q^*\|_2^2 <0\\
&\Longleftrightarrow& b>b^*\Big[1+ \dfrac{\|Q^*\|^2_2}{4\|\Delta Q^*\|^2_2} (\mu^2 + \dfrac{4\|\nabla Q^*\|^2_2}{\|Q^*\|^2_2} \mu )\Big] \\
&\Longleftrightarrow& b>b_*.
\end{eqnarray*}
So, given $\mu\in (-\dfrac{4\|\nabla Q^*\|^2_2}{\|Q^*\|^2_2},0)$, for any $b\in (b_*,b^*)$ we have
\begin{eqnarray*}
E_{\mu,b}(v_0)<-\dfrac{\mu^2}{8} %\quad v_0\in B_1,
\end{eqnarray*}
which implies 
\[\  m_{\mu,b}<-\dfrac{\mu^2}{8}\,.\]
%Thus \eqref{4.9} is verified.  
According to Lemma \ref{lm4.5}, $\{u_n\}$ is non-vanishing in $L^{2+\frac{8}{d}}$. Furthermore, by an interpolation inequality, \eqref{4.11} follows for all $q$ in $(2,\frac{2d}{(d-4)^+})$.  
This completes the proof. 
\end{proof}
\begin{remark} In view of Lemma \ref{lm4.1},
the constants $b^*$ is a sharp upper bound in the sense that $\mathcal{M}_{\mu,b}=\emptyset$
for any $b\ge b^*$. The analysis in the proof of Lemma \ref{lm4.6} also seems to suggest $b_*$ is a sharp lower bound, 
however, we are not able to verify this at present.  
\end{remark}

\begin{proof}[{\bf  Proof of Proposition \ref{prop4.1}}] 
We note from the proof of Proposition \ref{M}, Case $(ii)$ that to prove  any minimizing sequence of $m_{\mu,b}$ is pre-compact in $H^2$ and (\ref{mmb-E}) by the profile decomposition method, 
we only to show  that any minimizing sequence of $m_{\mu,b}$ is non-vanishing in the sense of \eqref{4.11}. 
However, according to Lemma \ref{lm4.5} and Lemma \ref{lm4.6},  \eqref{4.11} holds  under the assumptions on $\mu$ and $b$ in this proposition. 
Therefore, we have proven Proposition \ref{prop4.1}.
\end{proof}

\begin{proof}[{\bf Proof of Theorem \ref{th2}}]
 Let $\psi(t)$ be the solution of the Cauchy problem (\ref{BNLS})-(\ref{BNLS1}) with initial datum $\psi_0\in H^2$. By \eqref{4.5} and the conservation laws of energy and mass in Proposition \ref{local}, we deduce that for all $t\in I:=[0,T)$ %(the maximal existence interval)
\begin{eqnarray}\label{4.7}
\begin{array}{lll} \vspace{0.3cm}
 E_{\mu,b}(\psi_0)=E_{\mu,b}(\psi(t))\geq \dfrac{1}{2}\Big[1-\dfrac{b}{b^*} \Big]\|\Delta \psi(t)\|_2^2 + \dfrac{\mu }{2}\|\Delta \psi(t)\|_2\|\psi_0\|_2\,.
\end{array}
\end{eqnarray}
If $0<b<b^*$, then  \eqref{4.7} implies $\{\| \psi(t) \|_{H^2}\}$ is bounded for all $t\in I$. Thus from Proposition \ref{local} we know $\psi(t)$  exists globally in time.
In virtue of Proposition \ref{prop4.1}, it remains to show the stability of $\mathcal{M}_{\mu,b}$ by a standard contradiction argument as given in the proof of Theorem \ref{th1}.
Therefore the proof is complete.
\end{proof}

\begin{proof}[{\bf Proof of Theorem \ref{th22}}]  As we have remarked in the introduction, the proof of Theorem \ref{th22} is a straightforward technical translation verbatim from that of\\ \mbox{Theorem \ref{th2}.} 
\end{proof}
%\end{document}

%%Comparing the results in   The existence of the ground states for (\ref{BNLS}) was studied in \cite{BCMN, PoSte17} using a different method.
%Moreover, we provide in Theorem \ref{th1} (4) and Theorem \ref{th2}
% a bound $\lam_0$ and resp. $\lam_1, \lam_2$ in the mass-subcritical and critical cases.

\section{Concluding remarks}
The study of stable ground states solutions is a central problem for higher-order dispersive equations in the past few decades. 
Concerning the existence and stability theory for standing waves of fourth-order NLS (\ref{BNLS}) there have been growing activities in this field
 \cite{BaruchFibichMandelbaum2010,BCMN,BCGJ,Fibich2002,NataliPastor2015,PoSte17,Segata2010}, where was mainly considered the case $\mu\ge 0$ using different methods. %Lypunov functional  
Our primary contribution is to treat the technically more challenging case $\mu<0$ by constructing an orbitally stable set of g.s.s., which has filled the gap as elaborated in the introduction.  
Moreover, at the critical exponent $p=1+\frac8d$, 
we have essentially showed in Theorem \ref{th22} and  Lemma \ref{lm4.1} 
that $\norm{Q^*}_2$ %^{\frac8d}$ 
is the threshold for the existence of g.s.s. of (\ref{MinP}), or equivalently (\ref{BNLS})  for suitable $\mu<0$. 

The existence of minimizers for certain negative $\mu$ was partially studied in \cite{BonhNa15} using a different method constricted to 
the submanifold  $\{ \Vert u\Vert_{p+1}=1 \}$ in $H^2$, an equivalent of the Nehari manifold.  However, there were not revealed the admissible values of the mass levels or the relation between the range limit of $\mu$  and the g.s.s. $Q_p=Q(p,d)$. 
Also the stability issue were not available via the Nehari manifold method.  
The profile decomposition method we employed allows to address the existence problem for 
(\ref{BNLS}) or (\ref{1.3}) for both signs of $\mu$, which gives a simpler approach than e.g. \cite{BCMN} in the regime $p\in (1,1+\frac8d]$. 
We believe that the analysis in this paper provides certain optimal ranges for the parameter $\mu$ 
with both signs as shown in \mbox{Theorem \ref{th1}} and Theorem \ref{th22}. 
The profile decomposition also allows to study the stability 
and instability problem on a deeper level in the regime $p\ge 1+\frac8d$. 
In this respect, one can find in  \cite{BL2015} some closely related open question in the case $\mu<0$,  
in particular at the threshold level $Q^*$ if $p=1+\frac8d$, comparing \cite{Weinstein1986} for the corresponding paradigm for the classical NLS.
We will continue to investigate this model in a sequel to this work.  
The variational analysis elaborated in this paper and \cite{ZhangZhengZhu2019,ZhuZhaYa2011,zhu2016}
could potentially lead to sharper and more accurate descriptions of the asymptotic behaviors for the solitary waves by incorporating some of the spectral information for the associated linearized operators around the ground state down the path \cite{BF2011,BCMN,BCGJ,BL2015,NataliPastor2015,Weinstein1986}. 

\vspace{0.5cm}

\noindent\textbf{Acknowledgments}\; T.-J. Luo is partially supported by NSFC 11501137 and GDNSFC 2016A030310258.  S.-H. Zhu is partially supported by NSFC  11501395.  S.-J. Zheng would like to thank Atanas Stefanov and Kai Yang for helpful comments in their communications.

\end{document}